\documentclass[10pt,twoside]{amsart}
\usepackage{amssymb, amsmath, mathrsfs, mathtools, amsfonts, amsthm, amscd, yfonts}
\usepackage{xcolor}
\usepackage{hyperref} 
\usepackage{multicol}

\theoremstyle{plain}
\newtheorem{Theorem}{Theorem}
\newtheorem{Proposition}[Theorem]{Proposition}
\newtheorem{Lemma}[Theorem]{Lemma}
\newtheorem{Corollary}[Theorem]{Corollary}

\theoremstyle{definition}
\newtheorem{Definition}[Theorem]{Definition}
\theoremstyle{remark}
\newtheorem{Remark}[Theorem]{Remark}
\newtheorem{Example}[Theorem]{Example}

\begin{document}

\title[Regular gradings]{On finite dimensional regular gradings}
\author[L. Centrone]{Lucio Centrone}
\address{Dipartimento di Matematica, Universit\`a degli Studi di Bari Aldo Moro, Via Edoardo Orabona, 4, 70125 Bari, Italy}
\email{lucio.centrone@uniba.it}
\thanks{L. Centrone was partially supported by PNRR-MUR PE0000023-NQSTI}
\author[P. Koshlukov]{Plamen Koshlukov}
\address{IMECC, UNICAMP, Rua S\'ergio Buarque de Holanda 651, 13083-859 Campinas, SP, Brazil}
\email{plamen@unicamp.br}
\thanks{P. Koshlukov was partially supported by FAPESP Grant 2024/14914-9 and by CNPq Grant 307184/2023-4}
\author[K. Pereira]{Kau\^e Pereira}
\address{IMECC, UNICAMP, Rua S\'ergio Buarque de Holanda 651, 13083-859 Campinas, SP, Brazil}
\email{k200608@dac.unicamp.br}
\thanks{K. Pereira was supported by FAPESP Grant 2023/01673-0}



\begin{abstract} 
Let $A$ be an associative algebra over an algebraically closed field $K$ of characteristic 0. A decomposition $A=A_1\oplus\cdots \oplus A_r$ of $A$ into a direct sum of $r$ vector subspaces is called a \textsl{regular decomposition} if, for every $n$ and every $1\le i_j\le r$, there exist $a_{i_j}\in A_{i_j}$ such that $a_{i_1}\cdots a_{i_n}\ne 0$, and moreover, for every $1\le i,j\le r$ there exists a constant $\beta(i,j)\in K^*$ such that $a_ia_j=\beta(i,j)a_ja_i$ for every $a_i\in A_i$, $a_j\in A_j$. We work with decompositions determined by gradings on $A$ by a finite abelian group $G$. In this case, the function $\beta\colon G\times G\to K^*$ ought to be a bicharacter. A regular decomposition is {minimal} whenever for every $g$, $h\in G$, the equalities $\beta(x,g)=\beta(x,h)$ for every $x\in G$ imply $g=h$. In this paper we describe completely the structure of the finite dimensional algebras $A$ (with unit) admitting a $G$-regular grading. Moreover, we compute the graded codimension sequence for a class of such algebras assuming complete support and minimal regular decomposition. It turns out that, for these algebras, the graded PI-exponent coincides with the ordinary (ungraded) PI-exponent. Finally, we show that the regular decomposition of a finite-dimensional algebra $A$ with a regular $G$-grading is minimal if and only if $\exp(A)=|G|$.
\end{abstract}

\maketitle

\section{Introduction} 
The theory developed by A. Kemer in 1984--1986 is one of the most important achievements in the theory of algebras with polynomial identities, see \cite{kemer}. This theory produced a classification of the ideals of identities (also called \text{\it T-ideals}) of associative algebras in characteristic 0; it depends on the so-called \text{\it T-prime} T-ideals. The latter were described as the ideals of concrete, and well-known algebras. Kemer otained as a consequence the positive solution of the long-standing Specht problem: Is every T-ideal in characteristic 0 finitely generated as a T-ideal? Kemer's theory relies heavily on methods based on $\mathbb{Z}_2$-graded algebras. Since Kemer's work there has been a significant interest in group gradings on algebras and their graded polynomial identities. It should be mentioned that long before Kemer's work, the classification of the finite dimensional $\mathbb{Z}_2$-graded simple algebras was obtained by Wall \cite{wall}. These turned out to be closely related to the T-prime T-ideals. 

Considering algebras with an additional structure like group grading, or trace, or involution, or derivation, is sometimes easier. For example, if an algebra $A$ is $G$-graded then the homogeneous components tend to be "smaller" than the whole algebra, and thus may be easier to study. To exemplify this phenomenon, let us recall that the polynomial identities of a given algebra $A$ are known in very few cases: the matrix algebras $M_n(K)$, $n\le 2$ (\cite{razmm2, drenskym2, pkm2}), the infinite dimensional Grassmann algebra $E$ (\cite{latyshev, krreg}), $E\otimes E$ (\cite{popov}), the upper triangular matrix algebras $UT_n(K)$. On the other hand, the graded identities for the matrix algebras $M_n(K)$ are known, for every $n$, for the natural grading by $\mathbb{Z}_n$, see \cite{vasilovsky, azevedo}. 

In 2005, Regev and Seeman introduced the notion of a \text{\it regular decomposition} of an algebra, see \cite{regevz2}. Let $A=A_1\oplus \cdots\oplus A_r$ be a decomposition of $A$ into a direct sum of vector subspaces. Suppose that for every choice of indices $1\le i_j\le r$, $j=1$, \dots, $n$, there exist $a_{i_j}\in A_{i_j}$ such that $a_{i_1} \cdots a_{i_n}\ne 0$, and that for every $1\le i,j\le r$ there exists a constant $0\ne \beta(i,j)\in K$ such that $a_ia_j=\beta(i,j)a_ja_i$ for every $a_i\in A_i$, $a_j\in A_j$. Then the decomposition of $A$ is called a regular one. We are interested in decompositions of $A$ that are determined by a grading by a finite abelian group. Let us recall what a group grading on $A$ is. If $A$ is an algebra and $G$ a group then a vector space decomposition $A=\oplus_{g\in G} A_g$ is called a $G$-grading on $A$ whenever $A_gA_h\subseteq A_{gh}$ for every $g$, $h\in G$. Suppose $G$ is a finite group of order $k$ and that the regular decomposition of $A$ is determined by a $G$-grading on $A$. The $k\times k$ matrix $M^A=(\beta(g,h))$ is the \textsl{decomposition matrix} of $A$. In \cite{regevz2} it was proved that $M^A$ determines the multilinear polynomial identities of $A$. This implies that in characteristic 0, it determines the T-ideal of $A$. The notion of a \textsl{minimal decomposition} of $A$ was introduced and studied in the papers \cite{bahturin2009graded, bcommutation}. It is a regular decomposition such that the equalities $\beta(x,g)=\beta(x,h)$ for every $x\in G$ imply $g=h$. In other words, the matrix $M^A$ has no two equal columns. All this informally says that the decomposition is minimal whenever it cannot be "coarsened" by means of joining two (or more) of the vector subspaces in it. We recall here that the group grading on $A$ and the associativity of the product in $A$ imply that $\beta\colon G\times G\to K^*$ is a bicharacter on $G$ with values in $K$. 

The authors of the paper \cite{bcommutation} considered $M_n(K)$, the matrix algebra of order $n$ with its natural grading by the group $\mathbb{Z}_n\times \mathbb{Z}_n$. They computed the determinant of the matrix $M^A$, it turned out to be equal to $\pm n^{n^2}$. In \cite{bahturin2009graded} it was conjectured that the decomposition matrix is invertible if and only if the corresponding regular decomposition is minimal. Furthermore, the authors of \cite{bahturin2009graded} asked whether the determinant of the decomposition matrix and the number of direct summands in a minimal (regular) decomposition are invariants of the algebra $A$. All this was shown to be true by Aljadeff and David in \cite{Eli1} under the assumption that $K$ is of characteristic 0. In a recent papers of ours we showed that such conjectures do not hold if the base field is of characteristic $p>2$ \cite{LPK}. Let us point out that in \cite{Eli1} the authors proved that for a regular minimal decomposition of $A$, determined by the grading by a group $G$, one has $\det M^A=\pm |G|^{|G|/2}$. Furthermore, the authors of \cite{Eli1} proved that $|G|$ equals the PI exponent of $A$. Recall that the PI exponent of a PI algebra is one of the most important numerical invariants of $A$, and it has been extensively studied for more than fourty years. Let $K\langle X\rangle$ be the free associative algebra freely generated over $K$ by the countable set $X=\{x_1,x_2,\ldots\}$, and let $I$ be the ideal of identities of the algebra $A$. Denote by $P_n$ the subspace of $K\langle X\rangle$ of all multilinear polynomials in the variables $x_1$, \dots, $x_n$, it has a basis consisting of all monomials $x_{\sigma(1)} \cdots x_{\sigma(n)}$ where $\sigma\in S_n$, the symmetric group on $\{1,\ldots, n\}$. In characteristic 0, the T-ideal $I$ is determined by its multilinear elements, that is, by the vector spaces $P_n\cap I$. Hence one may study the identities of $A$ by considering the intersections $P_n\cap I$. But in the 1970s, A. Regev \cite{regev} proved a remarkable theorem. Denote by $P_n(A)=P_n/(P_n\cap I)$, and by $c_n(A)=\dim P_n(A)$, the $n$-th codimension of $A$. Regev's theorem states that if $A$ satisfies an identity of degree $d$ then $c_n(A)\le (d-1)^{2n}$. Since $\dim P_n=n!$ this means that $P_n\cap I$ is much "larger" than $P_n(A)$, and the latter vector space should be easier to study than the former. The codimension sequence of an algebra gives an estimate how fast its identities grow. In 1999 Giambruno and Zaicev proved that if $A$ is a PI algebra the limit $\lim_{n\to\infty} (c_n(A))^{1/n}$ exists and is a non-negative integer, thus confirming a conjecture posed by Amitsur, see for example the monograph \cite{GZbook}. 

In this paper we give a complete description of the finite dimensional $G$-graded regular algebras assuming $G$ a finite abelian group and the base field $K$ algebraically closed and of characteristic 0. In order to achieve these we use the graded version of the  Wedderburn--Malcev decomposition of a finite dimensional algebra, and reduce the study to that of twisted group algebras. As a consequence of the description above, we also compute the graded codimension sequence of algebras $A$ with a regular grading whose regular decomposition is minimal, has complete support, and whose graded simple components have nonzero product with the Jacobson radical. It turns out that, in this case, the graded PI-exponent of $A$ coincides with the usual (ungraded) one. Finally, we provide a characterization of finite-dimensional $G$-graded regular algebras with minimal decomposition in terms of their PI-exponent.

\section{Preliminares}
We recall the main notions concerning the so-called \textsl{regular gradings}. For more details,
see \cite{LPK}, \cite{bahturin2009graded}, or \cite{regevz2}. Throughout this paper, we denote by $K$ an algebraically closed field of characteristic $0$, $G$ a finite abelian group, and $X = \{x_{1}, x_{2},\ldots, x_{n},\ldots \}$ an infinite countable set of variables. The word \textsl{algebra} means an associative $K$-algebra.

\begin{Definition}
\label{def.reg.}
Let $A$ be a $G$-graded algebra. We say the algebra $A$ is $G$-\textsl{graded regular} if the following conditions hold:
\begin{enumerate}
      \item[(i)] for every $n\in \mathbb{N}$ and every $n$-tuple $(g_{1},\ldots, g_{n})\in G^{n}$, there exist homogeneous elements $a_{1}\in A_{g_{1}}$, \dots, $a_{n}\in A_{g_{n}}$ such that $a_{1}\cdots a_{n}\neq 0$.
     \item[(ii)] for every $g$, $h\in G$ and for every $a_{g}\in A_{g}$, $a_{h}\in A_{h}$, there exists $\beta(g,h)\in K^{\ast}$ satisfying
     \[
     a_{g}a_{h}=\beta(g,h)a_{h}a_{g}.
     \]
\end{enumerate}
 Moreover, we define the {\sl regular decomposition matrix} associated with the regular decomposition of the $G$-graded algebra $A$ as $M^{A}:=(\beta(g,h))_{g,h}$.
\end{Definition}

Let $A$ be a $G$-graded algebra with a regular grading. Notice that from item $(ii)$ of the definition above we get a function $\beta \colon G\times G \rightarrow K^{\ast}$ satisfying the following conditions. For every $g$, $h$, $k\in G$ one has:
\begin{equation}
\label{eq.1}
\beta(g,h)=(\beta(h,g))^{-1}
\end{equation}
\begin{equation}
    \label{eq.2}
   \beta(g,h+k)=\beta(g,h)\beta(g,k) 
\end{equation}
\begin{equation}
    \label{eq.3}
    \beta(g+k,h)=\beta(g,h)\beta(k,h)
\end{equation}
We say $\beta$ is the {\it bicharacter} associated to the regular decomposition of $A$.

Besides we also have
\[
\operatorname{Supp}(A)=\{g\in G\mid A_{g}\neq 0\}=G.
\]

\begin{Definition} A function $\nu\colon G\times G\rightarrow K^{\ast}$ satisfying (\ref{eq.1}), (\ref{eq.2}), and (\ref{eq.3}), is called a {\it bicharacter} of $G$.     
\end{Definition}

The next are classical examples of regular gradings.

\begin{Example}
\label{Grass.exe}
Let $E$ be the infinite dimensional Grassmann algebra. Then, it can be verified that $E$ with the natural $\mathbb{Z}_{2}$-grading $E=E_{0}\oplus E_{1}$ is a $\mathbb{Z}_{2}$-graded algebra with regular grading whose corresponding bicharacter is given by: $\beta(0,0)=\beta(0,1)=\beta(1,0)=1$ and $\beta(1,1)=-1$.
\end{Example}
\begin{Example}
\label{Pauli.grading.exe}
Given $n \in \mathbb{Z}$, let $\xi \in K$ be a primitive $n$-th root of unity. We can consider on $M_n(K)$ the $\xi$-grading \cite[Example 3.1.7]{GZbook}, which is defined by the matrices $X:=\operatorname{diag}(\xi^{n-1}, \xi^{n-2}, \dots, \xi, 1)$ and $Y:=e_{n,1}+\sum_{i=1}^{n-1} e_{i,i+1}$. Then,
    \[
        M_{n}(K)=\oplus_{1 \leq i,j \leq n-1} \operatorname{span}_K\{ X^i Y^j \}
    \]
defines on $M_n(K)$ a structure of a regular $(\mathbb{Z}_{n}\times \mathbb{Z}_{n})$-grading with bicharacter given by $\beta(i,j)=\xi^{jk-il}$. Moreover $\det (M^{M_{n}(K)})=\pm n^{n^{2}}$. 
\end{Example}
\begin{Example} Let $A$ be endowed with a regular $(\mathbb{Z}_n \times \mathbb{Z}_n)$-grading with bicharacter $\beta$. It is well known that $\beta((0,1),(1,0))$ is an $n$-th root of unity. Furthermore, if $\beta((0,1),(1,0))$ is primitive, then $\det M^{A}=\pm n^{n^2}$ \cite[Theorem 20]{LPK}.
\end{Example}
\begin{Example} Denote by $H^{2}(G,K^{\ast})$ the second cohomology group of $G$ with values in $K^*$  ($G$ acts trivially on $K^{\ast}$). The elements $\tau\in H^{2}(G,K^{\ast})$ satisfy the \textsl{cocycle condition} 
 \[
 \tau(g,h+k)\tau(h,k)=\tau(g+h,k)\tau(g,h),\quad \text{for every}\quad g,h,k\in G.
 \]
Consider the twisted group algebra $K^{\tau}G$ with basis $\{X_{g}\mid g\in G\}$, satisfying $X_{g}X_{h}=\tau(g,h)X_{g+h}$ \cite[Example 8]{LPK}. Then, $K^{\tau}G$ is a finite dimensional $G$-graded algebra with regular grading and bicharacter given by
\[
\beta(g,h)=\tau(g,h)\tau^{-1}(h,g),\quad \text{for any}\quad g,h\in G.
\]
 In this case we say $\tau$ {\it induces} $\beta$.  
\end{Example}


\begin{Example}
\label{matrix.direct}
If $A$ and $B$ are two $G$-graded algebras having a regular grading with the same bicharacter $\beta$, then $A\oplus B$ with the grading: $(A\oplus B)_{g}:=A_{g}\oplus B_{g}$, $g\in G$, is endowed with a regular  $G$-grading and bicharacter $\beta$ \cite[Lemma 26]{Eli1}.
\end{Example}

Recall that a $G$-graded algebra $D$ is called a \textsl{$G$-graded division algebra} if every homogeneous element of $D$ is invertible. For instance, by construction, a twisted group algebra is a graded division algebra.

\begin{Theorem}
\label{div.algebra.twisted}
\cite[Theorem 2.13]{elduque2013gradings} Let $D$ be a finite dimensional $G$-graded division algebra. Then, $D$ is isomorphic to a twisted group algebra. 
\end{Theorem}
\begin{Definition}
\label{def.min}
Given a regular $G$-grading on $A$ with bicharacter $\beta$, the regular decomposition of $A$ is called \textsl{nonminimal} if there exist $g$, $h\in G$ such that $\beta(x,g)=\beta(x,h)$ for every $x\in G$. Otherwise we say the regular decomposition of $A$ is \textsl{minimal}.
\end{Definition}

Bahturin and Regev conjectured in \cite{bahturin2009graded} a regular $G$-graded algebra has a minimal regular decomposition if and only if $\det M^A \neq 0$. Aljadeff and David provided a positive answer to this conjecture in \cite{Eli1} in the case where the field $K$ is algebraically closed and of characteristic $0$. However, the conjecture fails when $\text{char}(K)=p > 2$ as it was shown in \cite{LPK}. 

Let us recall one of the results of \cite{Eli1}.

\begin{Theorem}\cite[Theorem 7]{Eli1}
\label{charac.minimal}
Let $A$ be a $G$-graded algebra with regular grading over a field $K$ with bicharacter $\beta\colon G\times G\rightarrow K^{\ast}$. Then we have
\begin{enumerate}
    \item If $\det M^{A}\neq 0$, then the regular decomposition of $A$ is minimal.
    \item Additionally, if $K$ is algebraically closed, then the regular decomposition of $A$ is minimal if and only if $\det M^{A}\neq 0$. Also, $\exp(A)=|G|$.
\end{enumerate}
\end{Theorem}

\begin{Example} 
Here we show two examples of minimal regular decompositions. 
\begin{enumerate}
    \item[(a)] Consider $E$ the Grassmann algebra of Example \ref{Grass.exe}. Then, since 
    \[
    \det M^{E}=\det\begin{pmatrix}
        1 & 1\\
        1 & -1
    \end{pmatrix}=-2,
    \]
    we conclude that the regular decomposition of $E$ is minimal. 
    \item[(b)] The regular $(\mathbb{Z}_{n}\times\mathbb{Z}_n)$-grading $M_{n}(K)$ of Example \ref{Pauli.grading.exe} has minimal regular decomposition because $\det M^{M_{n}(K)}=\pm n^{n^{2}}$.
\end{enumerate}
\end{Example}
Let  $\beta\colon G\times G\rightarrow K^{\ast}$ be a bicharacter of $G$. Consider the free $G$-graded algebra $K\langle X_{G}\rangle$ $($see \cite[Definition 3.3.3]{GZbook}$)$. Recall that the homogeneous degree of a monomial $w=x_{i_{1}}^{(g_{1})} \cdots x_{i_{n}}^{(g_{n})}\in K\langle X_{G}\rangle$ is defined as $\deg (w)=g_{1}+\cdots+g_{n}$.  Consider the following family of polynomials 
\[
x_{i}^{(g)}x_{j}^{(h)}-\beta(g,h)x_{j}^{(h)}x_{i}^{(g)},\quad \text{where}\quad g,h\in G,\quad i,j\in \mathbb{N}
\]
and denote by $I$ the $T_{G}$-ideal of $K\langle X_{G}\rangle$ generated by this family $($see \cite[Definition 3.3.4]{GZbook} for the definition of $T_{G}$-ideal $)$. 
\begin{Theorem} 
\label{free.reg.}
The algebra $R:=K\langle X_{G}\rangle/I$ has a regular $G$-grading with bicharacter $\beta$.
\end{Theorem}

 The algebra $K\langle X_{G}\rangle/I$ will be called {\it relatively free} $G$-graded algebra with bicharacter $\beta$. Of course, its grading is regular.

\section{A motivation: \texorpdfstring{regular $(\mathbb{Z}_{2}\times \mathbb{Z}_{2})$-gradings}{[(Z2 x Z2)-graded algebras with regular grading}}

The main goal of this section is to describe the structure of finite dimensional $(\mathbb{Z}_{2} \times \mathbb{Z}_{2})$-graded algebras with regular grading assuming the neutral component is isomorphic to $K$. It is important to emphasize that, from now on, all associative algebras considered in this paper will be assumed to have a unit $1 \in A$, unless stated otherwise.

\begin{Remark} 
\label{remark.inverse}
Let $A$ be a finite dimensional algebra with unit. Given $0\neq x\in A$, if there exists $y\in A$ such that $xy=1$, then $yx=1$, that is, $y$ is the inverse of $x$. Indeed, since $xy=1$, then $xA=A$. Consider the operator $R_{x}\colon A\rightarrow A$ given by $R_{x}(a)=xa$, and remark if $R_{x}(a)=0$, then $xa=0$.  Thus
	\[
	a=axy=(ax)y=0, 
	\]
	and since $\dim_{K} A<\infty$, it follows that $R_{x}$ has an inverse. In particular, there exists $z\in A$ such that $zx=1$, then $y=z$. 
\end{Remark}

The following proposition is the key motivation for the entire paper:

\begin{Proposition}
Let $A$ be a finite dimensional algebra with a regular $(\mathbb{Z}_{2} \times \mathbb{Z}_{2})$-grading and bicharacter $\beta$. Assume that $A_{(0,0)}=K$. Then, there exists a cocycle $\alpha \in H^{2}(\mathbb{Z}_{2} \times \mathbb{Z}_{2}, K^{\ast})$ such that $A \cong K^{\alpha}(\mathbb{Z}_{2} \times \mathbb{Z}_{2})$, where $\alpha$ induces $\beta$.
\end{Proposition}
\begin{proof}
Let $k:=\dim A_{(0,1)}$, $l:=\dim A_{(1,0)}$, and $m:=\dim A_{(1,1)}$. Because of the regularity of $A$ and the fact that $A_{(0,0)}=K$, there exist $e_{1}\in A_{(0,1)}$, $e_{2}\in A_{(1,0)}$, and $e_{3}\in A_{(1,1)}$ such that $e_{1}e_{2}e_{3}=1$. According to Remark \ref{remark.inverse}, the left inverse of an element in $A$ exists if and only if its right inverse also exists; that is, $e_{1}^{-1}$, $e_{2}^{-1}$, and $e_{3}^{-1}$ do exist. Moreover, it can be verified 
\[
\begin{cases}
e_{1}^{-1}=e_{2}e_{3}\\
e_{2}^{-1}=\beta((0,1),(1,0))e_{1}e_{3}\\
e_{3}^{-1}=\beta((1,1),(1,1))e_{1}e_{2}
\end{cases}
\]
We will study now the multiplication table of the elements of the set 
\[
\{1\}\cup \{e_{j}^{\delta}\mid \delta\in \{\pm 1\},\quad 1\leq j\leq 3\}.
\]
There are several options to be considered:

\vspace{0.2 cm}

 1) $e_{i}^{-1} \notin \text{\rm span}_{K}\{e_{i}\}$, for all $1 \leq i \leq 3$. In this case, since $e_{1}^{2} \in A_{0}$, we have $e_{1}^{2} = 0$, as otherwise $e_{1}$ would be a scalar multiple of $e_{1}^{-1}$, leading to a contradiction. However, since $e_{1}$ is invertible, the equality $e_{1}^{2}=0$ would imply that $e_{1}=0$, that is a contradiction.


\vspace{0.2 cm}

2) $e_{i}=e_{i}^{-1}$, $e_{j}=e_{j}^{-1}$, and $e_{k}^{-1} \notin \text{\rm span}_{K}\{e_{k}\}$, for some distinct indices $i$, $j$, and $k$ in $\{1,2,3\}$. Without loss of generality, suppose $e_{1}^{-1}=e_{1}$, $e_{2}^{-1}=e_{2}$, and $e_{3}^{-1} \neq e_{3}$. By the same argument as in the previous case, we would have $e_{3}^{2}=0$, which is a contradiction since $e_{3}$ is invertible.

\vspace{0.2 cm}

3) Two elements are not in the space spanned by their inverses. The conclusion in this case is exactly the same as in item 1).

\vspace{0.2 cm}

4) $e_{1}^{-1}=e_{1}$, $e_{2}^{-1}=e_{2}$, and $e_{3}^{-1}=e_{3}$. In this case, if $B$ is the $(\mathbb{Z}_{2}\times \mathbb{Z}_{2})$-graded algebra, with regular grading, generated by $\{1,e_{1},e_{2},e_{3}\}$, then its multiplication is given in the table below:
\begin{center}
	\begin{tabular}{c|c c c c}
		$\times$ & 1 & $e_{1}$ & $e_{2}$ & $e_{3}$ \\ \hline
		1       & 1  & $e_{1}$  & $e_{2}$  & $e_{3}$  \\
		$e_{1}$ & $e_{1}$  & 1  & $e_{3}$  & $\beta((1,0),(0,1))e_{2}$  \\
		$e_{2}$ & $e_{2}$  & $\beta((1,0),(0,1))e_{3}$  & 1  & $e_{1}$ \\
		$e_{3}$ & $e_{3}$  & $e_{2}$  & $\beta((1,1),(1,0))e_{1}$ & 1 \\
	\end{tabular}
\end{center}
Therefore there exists $\alpha\in H^{2}(\mathbb{Z}_{2}\times \mathbb{Z}_{2},K^{\ast})$ such that $B\cong K^{\alpha}(\mathbb{Z}_{2}\times \mathbb{Z}_{2})$. Moreover, since $\beta$ is the bicharacter of $B$, it follows
\[
\beta(g,h)=\alpha(g,h)\alpha(h,g)^{-1},\quad \text{for all}\quad g, h\in \mathbb{Z}_{2}\times \mathbb{Z}_{2}.
\]
Now, take a basis of $A$ formed by homogeneous elements such that
\[
\mathcal{A}=\{1,e_{1},e_{2},e_{3}\}\cup \mathcal{B}_{(0,1)}\cup \mathcal{B}_{(1,0)}\cup \mathcal{B}_{(1,1)},
\]
where 
	\[
	\mathcal{B}_{(0,1)}:=\{a_{j}^{(0,1)}\mid 1\leq j\leq k-1\}
	\]
	\[
	\mathcal{B}_{(1,0)}:=\{a_{j}^{(1,0)}\mid 1\leq j\leq l-1\}
	\]
	\[
	\mathcal{B}_{(1,1)}:=\{a_{j}^{(1,1)}\mid 1\leq j\leq m-2\}.
	\]
By the uniqueness of the inverse of $e_{1}$, we have
\[
\begin{cases}
e_{1}a_{i}^{(0,1)}=0\\
e_{1}a_{j}^{(1,0)}=0\\
e_{1}a_{k}^{(1,1)}=0
\end{cases}
\]
and multiplying all the above equations on the left by $e_{1}$, it follows that for all $i$, $j$ and $k$, we get
\[
a_{i}^{(0,1)}=a_{j}^{(1,0)}=a_{k}^{(1,1)}=0.
\]
 As a consequence, we have $B=A$. Therefore, by Theorem \ref{div.algebra.twisted}, there exists a cocycle $\alpha \in H^{2}(\mathbb{Z}_{2} \times \mathbb{Z}_{2}, K^{\ast})$ such that $A \cong K^{\alpha}(\mathbb{Z}_{2} \times \mathbb{Z}_{2})$.
\end{proof}

\section{Description of the structure of finite dimensional regular gradings with complete support }

We begin this section with a simple yet rather subtle remark. Before proceeding, recall that a $G$-graded algebra $C$ is called $\beta$-\textsl{commutative} if for every $g$, $h \in G$, and $c_{g} \in C_{g}$ and $c_{h} \in C_{h}$, we have $c_{g}c_{h}=\beta(g,h)c_{h}c_{g}$.

\begin{Remark}
\label{trivial.case}
Suppose $G$ is the trivial abelian group $\{0\}$. In this case, a bicharacter $\beta\colon G\times G\rightarrow K^{\ast}$ is necessarily trivial, that is, $\beta(0,0)=1$. Hence, any $G$-graded $\beta$-commutative algebra must be commutative. Therefore, any nonzero $\beta$-commutative algebra $A$ with grading $A_{0}=A$ is trivially regular whose regular decomposition is minimal because $M^{A}=(1)$. Conversely, every algebra with a regular grading over $G=\{0\}$ must be a nonzero commutative algebra and automatically it has a minimal regular decomposition. 
\end{Remark}
Because of the above remark, in this and in the subsequent sections of the paper,  $G$ will denote a non-trivial finite abelian group, i.e. $|G|>1$. In this section we will describe algebras with regular gradings with a restrictive condition, namely \textsl{complete support}.

\begin{Proposition}
\label{first.regular}
Let $A$ be a finite dimensional $G$-graded algebra with regular grading. If $A_{0}=K$, then the algebra $A$ is isomorphic to a twisted group algebra. 
\end{Proposition}
\begin{proof} 
    For every $g\in G$, we set $k_{g}:=\dim A_{g}$. Given the pair $(g,-g)\in G^{2}$, $g\in G\setminus\{0\}$, by regularity there exist $a_{g}\in A_{g}$ and $a_{-g}\in A_{-g}$ such that $a_{g}a_{-g}\neq 0$. Since $A_{0}=K$, without loss of generality, up to a nonzero scalar, we can assume  $a_{g}a_{-g}=1$.

For each $g \in G$, where $k_{g}>1$, we can choose a basis $\mathcal{B}_{g}$ of $A_{g}$ that includes $a_{g}$, that is, there exist $b_{g,j} \in A_{g}$, $1 \leq j \leq k_{g}-1$, such that
\[
\mathcal{B}_{g}=\{a_{g}\}\cup \mathcal{W}_{g},\quad \text{where}\quad \mathcal{W}_{g}=\{b_{g,j}\mid 1\leq j\leq k_{g}-1\}
\]
is a basis of $A_{g}$. If $k_{g}=1$ we simply set $\mathcal{B}_{g}=\{a_{g}\}$.

It follows that the disjoint union
\[
\mathcal{A}=\bigcup_{g\in G}\mathcal{B}_{g}
\]
is a homogeneous basis of $A$. Now, consider the set
\[
\Omega:=\{g\in G\setminus\{0\}\mid k_{g}>1\}.
\]
Suppose that $\Omega\neq \emptyset$, and let $h\in \Omega$. Given $1\leq j\leq k_{h}-1$, consider $b_{h,j}\in \mathcal{W}_{h}$; of course, $a_{-h}b_{h,j}\in A_{0}$ and since $b_{h,j}$ and $a_{h}$ are linearly independent and $a_{h}$ is invertible, it follows  that $a_{-h}b_{h,j}=0$. By the fact that $a_{-h}$ is invertible we have $b_{h,j}=0$.


By the arbitrariness of $1\leq j\leq k_{h}-1$, we conclude that $\mathcal{W}_{h}=\{0\}$, so $\Omega=\emptyset$. It implies  that $\{a_{g}\mid g\in G\}$ is a basis of $A$, then by Theorem \ref{div.algebra.twisted}, $A$ is isomorphic to a twisted group algebra, that is, there exists a cocycle $\alpha \in H^{2}(G,K^{\ast})$ which induces the bicharacter of $A$ and $A\cong K^{\alpha}G$.
\end{proof}

Let $G=\{0,g_1, \ldots,g_k\}$ be a finite abelian group. A theorem of Miller, proved in 1903, \cite{Miller}, states that $g_1+\cdots+g_k=0$ unless $G$ contains exactly one element of order 2. In the following example we put the condition $g_1+\cdots+g_k=0$. While this may seem to be a restriction, it turns out it is harmless. This is due to the fact that the groups we consider either have no elements of order 2 or have more than one such element.

\begin{Proposition}
\label{parallel.first}
Let $G=\{0,g_{1},\ldots, g_{k}\}$ and let $g_{1}+\cdots +g_{k}=0$. Suppose that $A$ is a $G$-graded $\beta$-commutative algebra with $A_{0}=K$. Assume that  $x_{i_{1}}^{(g_{1})}\cdots x_{i_{k}}^{(g_{k})}\notin T_{G}(A)$, then $A$ is a twisted group algebra.  
\end{Proposition}

\begin{proof}
    Since $x_{i_{1}}^{(g_{1})}\cdots x_{i_{k}}^{(g_{k})}\notin T_{G}(A)$, there exist $a_{g_{1}}\in A_{g_{1}}$, \dots, $a_{g_{k}}\in A_{g_{k}}$ such that up to a constant, we can assume $a_{g_{1}}\cdots a_{g_{k}}=1$. It follows that there exist $\lambda_{g_{1}}$,\dots, $\lambda_{g_{k}}\in K$ such that 
\[
a_{g_{j}}^{-1}=\lambda_{g_{j}}a_{g_{1}}\cdots a_{g_{j-1}}\cdot \widehat{a_{g_{j}}}\cdot a_{g_{j+1}}\cdots a_{g_{k}},
\]
where $\widehat{a_{g_j}}$ represents the omission of the element $a_{g_j}$ in the product. Note that $a_{-g_{j}}=t_{j}a_{g_{j}}^{-1}$, for some $t_{j}\in K$. Indeed, since $a_{-g_{j}}\in A_{-g_{j}}$, we have $a_{-g_{j}}a_{g_{j}}\in A_{0}\setminus\{0\}$. Thus, since $A_{0}=K$, we get $a_{g_{j}}^{-1}=t_{j}a_{-g_{j}}$, for some $t_{j}\in K^{\ast}$. 

Considering the linearly independent set $\mathcal{C}=\{1,a_{g_{1}},\ldots, a_{g_{k}}\}$, we can complete $\mathcal{C}$ to a homogeneous basis $\mathcal{A}=\mathcal{C}\cup \Big(\bigcup_{j=1}^{k}\mathcal{B}_{j}\Big)$, where $\mathcal{B}_{j}=\{b_{g_{j},l}\mid 1\leq l\leq d_{j}-1\}$ 
and $d_{j}=\dim A_{g_{j}}$. 

We claim that $\mathcal{B}_{j}=\emptyset$ for every $1\leq j\leq k$. Note that if $\mathcal{B}_{j}\neq \emptyset$, for some $1\leq j\leq k$, then there exists $b_{g_{j},l}\neq 0$ in $\mathcal{B}_{j}$. Again, $a_{-g_{j}}b_{g_{j},l}=b_{g_{j},l}a_{-g_{j}}=0$, hence  $b_{g_{j},l}=0$, that is a contradiction. Consequently, $\mathcal{B}_{j} = \emptyset$ for all $1\leq j \leq k$; therefore, $\mathcal{A}=\mathcal{C}$ and $\dim A=k+1$. Now the proof follows by Theorem \ref{div.algebra.twisted}.
\end{proof}

We note that Proposition \ref{first.regular} is a particular case of Proposition \ref{parallel.first}. Furthermore, we also have the following easy consequence.

\begin{Corollary} Let $A$ be a $(\mathbb{Z}_{n}\times \mathbb{Z}_{n})$-graded $\beta$-commutative algebra satisfying the conditions of the last proposition. Then $A$ is a twisted group algebra.
\end{Corollary}

\begin{Definition}
Let $Q$ be a finite abelian group, and let $\alpha \in H^{2}(Q,K^{\ast})$ be a cocycle. We say that $x \in Q$ is regular if $\alpha(x,s) = \alpha(s,x)$ for every $s \in Q$. 
\end{Definition}
\begin{Lemma}
\label{charc.minimal.twisted}
\cite[Theorem 2.2]{Karpilovsky} 
Let $Q_{0}(\alpha)$ denote the subgroup of $Q$ consisting of all  regular elements. The following conditions are equivalent:
\begin{enumerate}
    \item $Q_{0}(\alpha) = \{0\}$.
    \item $K^{\alpha}Q$ is central simple. 
\end{enumerate}
\end{Lemma}
\begin{Remark}
\label{charc.minimal.twisted.1}
In the context of Lemma \ref{charc.minimal.twisted}, it follows from \cite[Proposition 30]{LPK} that the regular decomposition of $K^{\alpha}Q$ is minimal if and only if $Q_{0}(\alpha)=\{0\}$. 
    
\end{Remark}

Recall that if $R=\oplus_{g\in G}R_{g}$ is a $G$-graded algebra and $V$ is a $G$-graded vector space, then  $V$ is a left graded $R$-module if there exists a linear action of $R$ on $V$ satisfying $R_{g}V_{h}\subseteq V_{gh}$, for every $g$, $h\in G$. We define the $G$-graded Jacobson radical of $R$, denoted by $J^{\mathrm{gr}}(R)$, as the intersection of all annihilators of its graded irreducible left $R$-modules. It can be shown that $J^{gr}(R)$ is a $G$-graded ideal of $R$. Now, since $K$ is an algebraically closed field of characteristic $0$ and $|G| < \infty$, it follows that $J(R)$ is a homogeneous ideal of $R$, i.e, $J^{\mathrm{gr}}(R)=J(R)$ \cite[Theorem 4.4]{SusanMontgomery}. Additionally, by \cite[Corollary 4.2]{SusanMontgomery}, we have

\begin{equation}
\label{eq.radical}
    J(R_{0})=J(R)_{0}=R_{0}\cap J(R).
\end{equation}

\begin{Remark}
    Here we recall that it is not easy to prove that $J(R)$ is a homogeneous ideal. This has been established in some particular albeit important cases, see for example \cite{diniz2024gradings}. It is not known, as of now, whether $J(R)$ is homogeneous for an arbitrary grading on $R$. 
\end{Remark}

We recall an important result about $G$-graded simple algebras due to Bahturin, Zaicev and Sehgal.

\begin{Theorem}\cite[Theorem 3]{bahturin2005finite}
\label{bahturin.zaicev}
 Let $R$ be a finite dimensional $Q$-graded simple algebra where $Q$ is a finite group with identity element $e$. Then there exist a finite subgroup $H$ of $Q$, a cocycle $\alpha\in H^{2}(H,K^{\ast})$, an integer $r\geq 1$ and an $r$-tuple $(q_{1},\ldots, q_{r})\in Q^{r}$ such that $A$ is $Q$-graded isomorphic to $\Lambda=K^{\alpha}H\otimes M_{r}(K)$ (here $\Lambda_{q}=\text{span}_{K}\{X_{h}\otimes e_{i,j}\mid q=q_{i}^{-1}hq_{j}\}$). Moreover, $\{X_{h}\mid h\in H\}$ is the canonical basis of $K^{\alpha}H$ and $e_{i,j}\in M_{r}(K)$ is the $(i,j)$ elementary matrix.  In particular, the idempotents $X_{0}\otimes e_{i,i}$, as well as the identity of $A$, are homogeneous of degree $e\in Q$.
\end{Theorem}

We have the next result.

\begin{Proposition} 
\label{prop.aux}
Let $A$ be a finite dimensional $G$-graded simple $\beta$-commutative algebra.  Then $A$ is isomorphic to a twisted group algebra $K^{\alpha}H$, where $\alpha\in H^{2}(H,K^{\ast})$ and $H$ is a subgroup of $G$. 
\end{Proposition}
\begin{proof} By Theorem \ref{bahturin.zaicev}, the algebra $A$ is $G$-graded isomorphic to the $G$-graded algebra $\Lambda=K^{\alpha}H\otimes M_{r}(K)$. To prove the statement, it is enough to show that $r=1$. 


Suppose by contradiction that $r>1$ and let $1\leq i\neq j\leq r$. Then there exist $g$, $g'\in G$, $h$, $h'\in H$ such that $X_{h}\otimes e_{i,i}\in \Lambda_{g}$ and $X_{h'}\otimes e_{i,j}\in \Lambda_{g'}$. Notice that  
\[
(X_{h}\otimes e_{i,i})(X_{h'}\otimes e_{i,j})=X_{h}X_{h'}\otimes e_{i,i}e_{i,j}=\alpha(h,h')X_{h+h'}\otimes e_{i,j}\neq 0.
\]
On the other hand,
\[
(X_{h'}\otimes e_{i,j})(X_{h}\otimes e_{i,i})=X_{h'}X_{h}\otimes e_{i,j}e_{i,i}=0
\]
and this contradicts the $\beta$-commutativity of $\Lambda$. Then $r=1$, and so $H=\operatorname{Supp}(A)$. Therefore 
\[
\Lambda= K^{\alpha}H\otimes M_{r}(K)\cong K^{\alpha}H\otimes K\cong K^{\alpha}H
\]
 that is, $A\cong K^{\alpha}H$ (as $G$-graded algebras).
\end{proof}

Recall the characterization of the PI-exponent of $A$, denoted by $\exp(A)$.  
Consider the Wedderburn-Malcev decomposition of $A=B_{1}\oplus B_{2}\oplus \cdots \oplus B_{k}+J(A)$. Then, $\exp(A)$ is equal to the maximal value of the sum of the dimensions 
\[
\dim B_{i_{1}}+\cdots +\dim B_{i_{n}}
\]
where $B_{i_{1}}$,  \dots, $B_{i_{n}}$, are distinct and satisfy the condition
\[
B_{i_{1}}J(A)B_{i_{2}}J(A)\cdots J(A)B_{i_{n}}\neq 0. 
\]
We observe that, in the definition of the PI-exponent the dimensions of all simple subalgebras of $R$ are also taken into account.

Now we define what a complete support is for algebras with regular gradings.

\begin{Definition} Let $B$ be a finite dimensional $G$-graded algebra with regular grading. We say $B$ has \textsl{complete support} if any twisted group algebra contained in $B/J(B)$ is isomorphic to $ K^{\alpha}G$, for some $\alpha\in H^{2}(G,K^{\ast})$. 
\end{Definition}

\begin{Proposition}
\label{sum.twisted}
Let $A$ be a finite dimensional algebra with regular $G$-grading with bicharacter $\beta$. Then, there exist graded simple subalgebras $D_{1}$, \dots, $D_{k}$ of $A$, such that $D_{i}$ is isomorphic (as a graded algebra) to a twisted group algebra and $A=D_{1}\oplus \cdots \oplus D_{k}+J(A)$. Moreover, we have the following:
\begin{enumerate}
    \item[(1)] There exists $1\leq i\leq k$ such that $D_{i}$ is isomorphic to $K^{\alpha}G$, where $\alpha\in H^{2}(G,K^{\ast})$.
    
    \item[(2)] If $A$ has complete support, then
    \[
    A\cong K^{\alpha}G\oplus \cdots \oplus K^{\alpha}G+J(A)
    \]
     where $J(A)$ is the Jacobson radical of $A$, and the cocycle $\alpha$ induces $\beta$.   
\end{enumerate}
\end{Proposition}
\begin{proof}

 It follows from the graded Wedderburn-Malcev theorem \cite[Lemma 2.2]{bahturin2004identities} that the Jacobson radical $J=J(A)$ is graded with respect to the $G$-grading and there exists a semisimple subalgebra $D$ that is homogeneous in the $G$-grading such that  $A=D+J$ as a direct sum of vector subspaces. Moreover, $D$ can be decomposed as the direct sum $D=D_{1}\oplus D_{2}\oplus \cdots \oplus D_{k}$ of graded two-sided ideals of $D$ and each $D_{j}$ is a $G$-graded simple algebra. We also note that by construction $D_{i}D_{j}=0$ for $i\neq j$. By Proposition \ref{prop.aux} it follows that $D_{i}\cong K^{\alpha_{i}}H_{i}$ (as $G$-graded algebras), where $H_{i}=\operatorname{Supp}(D_{i})$ is a subgroup of $G$, for all $1\leq i\leq k$. Of course $J$ cannot have a regular grading because it is nilpotent. Therefore, $D$ must have a $G$-graded regular grading, and since $K^{\alpha_{i}}H_{i}$ it is not a $G$-graded regular algebra if $H_{i}$ is a proper subgroup of $G$, we conclude there exists $1\leq i\leq k$ such that $D_{i}\cong K^{\alpha_{i}}G$.

Now, suppose $A$ has complete support. In this case $D_{i}\cong K^{\alpha_{i}}G$, for all $1\leq i\leq k$. Given any $1\leq i\leq k$, let $\{X_{h}^{(i)}\mid h\in G\}$ be a basis of $K^{\alpha_{i}}G$. It is obvious that if $a_{g}:=\sum_{i=1}^{k}X_{g}^{(i)}$ and $a_{h}=\sum_{j=1}^{k}X_{h}^{(j)}$, then
 \[
 a_{g}a_{h}=\big(\sum_{i=1}^{k}X_{g}^{(i)}\big)\big(\sum_{j=1}^{k}X_{h}^{(j)}\big)=\sum_{i=1}^{k}\alpha_{i}(g,h)X_{g+h}^{(i)}
 \]
 since $X_{g}^{(i)}X_{h}^{(j)}=0$ if $i\neq j$. On the other hand we have
 \[
 a_{h}a_{g}=\big(\sum_{j=1}^{k}X_{h}^{(j)}\big)\big(\sum_{i=1}^{k}X_{g}^{(i)}\big)=\sum_{i=1}^{k}\alpha_{i}(h,g)X_{g+h}^{(i)}.
 \]
However we know that $a_{g}a_{h}=\beta(g,h)a_{h}a_{g}$, then for each $1\leq i\leq k$ we have
\[
\alpha_{i}(g,h)=\beta(g,h)\alpha_{i}(h,g)
\]
that is $\beta(g,h)=\alpha_{i}(g,h)\alpha_{i}(h,g)^{-1}$. Finally, we will show that $\alpha_{1} = \cdots = \alpha_{k} = \alpha$. It is well known (see \cite[Lemma 31]{Eli1}) that for every $1 \leq i \neq j \leq k$ one has 
\[
T_{G}(K^{\alpha_{i}}G)=T_{G}(K^{\alpha_{j}}G), 
\]
 since both $\alpha_{i}$ and $\alpha_{j}$ induce $\beta$. Suppose that $\alpha_{i} \neq \alpha_{j}$ for some $i \neq j$, then there exist $g$, $h \in G$ such that $\alpha_{i}(g, h) \neq \alpha_{j}(g, h)$. On the other hand, by definition, the polynomial 
\[
f = x_{m}^{(g)} x_{m_{1}}^{(h)} - \alpha_{i}(g, h) x_{m_{2}}^{(g+h)} \in K\langle X_{G}\rangle
\]
must be in $T_{G}(K^{\alpha_{i}}G)$, for every $m$, $m_{1}$, $m_{2} \in \mathbb{N}$. In particular, $f \in T_{G}(K^{\alpha_{j}}G)$. Now, consider the basis $\{X_{g}^{(j)} \mid g \in G\}$ of $K^{\alpha_{j}}G$, then
\[
f(X_{g}^{(j)}, X_{h}^{(j)}, X_{g+h}^{(j)})=0.
\]
 In other words,
 \[
 X_{g}^{(j)}X_{h}^{(j)}=\alpha_{i}(g,h)X_{g+h}^{(j)}.
 \]
  This means $\alpha_{j}(g,h)X_{g+h}^{(j)}=\alpha_{i}(g,h) X_{g+h}^{(j)}$, that is, $\alpha_{i}(g,h)=\alpha_{j}(g,h)$ which is a contradiction. It turns out $\alpha_{1}=\cdots=\alpha_{k}=\alpha$, and this completes the proof of (2).
 \end{proof}

Before the next result we make a simple but important remark.

\begin{Remark}
\label{Graded.subal.}
Let $A$ be a $G$-graded algebra (not necessarily unital) which is $\beta$-commutative. If there exists a graded subalgebra $C\subseteq A$ that has a regular grading, then $A$ also has a regular grading with bicharacter $\beta$. Indeed, we only need to verify the first condition of regularity. Given any $n$-tuple $(g_{1},\ldots, g_{n})\in G^{n}$ by regularity of $C$, there exists $c_{g_{1}}\in C_{g_{1}}$,\dots, $c_{g_{n}}\in C_{g_{n}}$ such that $c_{g_{1}}\cdots c_{g_{n}}\neq 0$, since these elements are homogeneous in $A$ we conclude $A$ is a $G$-graded regular algebra. 
\end{Remark}

Let $A$ be a finite dimensional algebra with regular $G$-grading and bicharacter $\beta$. By Proposition \ref{sum.twisted} we can write $A=D_{1}\oplus\cdots  \oplus D_{k}+J(A)$, where $D_{1}$,\dots, $D_{k}$ are twisted group algebras with at least one isomorphic to $K^{\alpha}G$, $\alpha\in H^{2}(G,K^{\ast})$. If $k=1$, since we are assuming $A$ has a unit it follows that $D_{1}J(A)\neq 0$. Suppose $k>1$ and assume $D_{1}J(A)=J(A)D_{1}=0$. Then, of course, in this case $A$ must be isomorphic to the algebra 
\[
\begin{pmatrix}
    B & 0\\
    0 & D_{1}
\end{pmatrix}
\]
where $B=D_{2}\oplus \cdots \oplus D_{k}+J(B)$, $J(B)=J(A)$. If $D_{1}\cong K^{\alpha}G$, then $B$ may be any $\beta$-commutative algebra, with or without a regular $G$-grading. If $D_{1}\cong K^{\alpha_{i}}H_{i}$, where $H_{i}$ is a proper subgroup of $G$, and $\alpha_{i}\in H^{2}(H_{i},K^{\ast})$, then we look at $D_{2}J(A)$ and this process can be continued. The next example illustrates a case in which this occurs.

\begin{Example}
\label{Example.Jacobson}
In $M_{6}(K)$, consider the following matrices
\[
L=\begin{pmatrix}
  D & 0 & 0\\
  0 & D & 0\\
  0 & 0 & D
\end{pmatrix},\quad M:=\begin{pmatrix}
  0 & N & 0\\
  0 & 0 & 0\\
  0 & 0 & 0
\end{pmatrix},\quad J=\begin{pmatrix}
  0 & 0 & 0\\
  0 & 0 & 0\\
  0 & 0 & N
\end{pmatrix} 
\]
where 
\[
D=\operatorname{diag}(1,-1),\quad \text{and}\quad N=\begin{pmatrix}
    0 & 1\\
    1 & 0
\end{pmatrix}. 
\]
These matrices satisfies the following relations
\[
LM=-ML,\quad M^{2}=0,\quad JM=MJ=0,\quad LJ=-JL.
\]

Moreover $N^{2}=I_{2}$,  $L^{2}=I_{6}$,  where $I_{2}$ and $I_{6}$ are the identity matrices in $2\times 2$ and $6\times 6$ respectively, $D^{n}=\operatorname{diag(1,(-1)^{n})}$ for all $n\in \mathbb{N}$, and 
\[
LJ=\begin{pmatrix}
  0 & 0 & 0\\
  0 & 0 & 0\\
  0 & 0 & P
\end{pmatrix},\quad \text{where}\quad P=\begin{pmatrix}
    0 & 1\\
    -1 & 0
\end{pmatrix}.
\]
Thus, $(LJ)^{2}=\operatorname{diag(0,0,I_{2})}$. Therefore if $\mathcal{A}$ is the subalgebra of $M_{6}(K)$ generated by $L$, $M$ and $J$, it can be seen that $\{I,J,J^{2},L,LJ,LJ^{2},M,LM\}$ is a basis of $\mathcal{A}$. Now, it is not difficult to see that $\mathcal{A}$ is a $\mathbb{Z}_{2}\times \mathbb{Z}_{2}$-graded algebra if we put
\[
\mathcal{A}_{(0,0)}:= \operatorname{span}_{K}\{I_{6},J^{2}\},\quad \mathcal{A}_{(0,1)}:=\operatorname{span}_{K}\{L,LJ^{2}\}, \quad \mathcal{A}_{(1,0)}:=\operatorname{span}_{K}\{M,J\}
\]
and $\mathcal{A}_{(1,1)}:=\operatorname{span}_{K}\{LM,LJ\}$.

Now, notice the following
\[
(I_{6})(L)(J)(LJ)=\begin{pmatrix}
  0 & 0 & 0\\
  0 & 0 & 0\\
  0 & 0 & -I_{2}
\end{pmatrix}
\]
which is non-nilpotent. Therefore, by the above relations, it follows that $\mathcal{A}$ is a $\mathbb{Z}_{2}\times \mathbb{Z}_{2}$-graded regular algebra whose the decomposition matrix is given by 
\[
M^{\mathcal{A}}=\begin{pmatrix}
    1 & 1 & 1 & 1\\
    1 & 1 & -1 & -1\\
    1 & -1 & 1 & -1\\
    1 & -1 & -1 & 1
\end{pmatrix}.
\]

Now, we will show that $J(\mathcal{A})=\operatorname{span}_{K}\{M,ML\}$. Consider $\mathcal{I}$ the ideal generated by $M$. By construction $\mathcal{I}=\operatorname{span}_{K}\{M,ML\}$, then it can be seen that
\[
\mathcal{A}/\mathcal{I}=\mathcal{B}
\]
where $\mathcal{B}$ is isomorphic to the subalgebra $\operatorname{span}_{K}\{I_{6},J,J^{2},L,LJ,LJ^{2}\}$. It can be easily seen that  $\operatorname{span}_{K}\{I_{2},P,N,D\}$ is $\langle D,N\rangle$, the subalgebra generated by $D$ and $N$, and it is isomorphic to $M_{2}(K)$. Thus, since $LJ^{2}=\operatorname{diag}(0,0,D)$, the subalgebra 
\[
\langle J,LJ^{2}\rangle=\operatorname{span}_{K}\{J,J^{2},LJ,LJ^{2}\}
\]

is isomorphic to $M_{2}(K)$. Now, defining $\mathscr{E}:=\operatorname{diag}(D,D,0)$, it follows that $\mathscr{E}^{2}=\operatorname{diag}(I_{2},I_{2},0)=I_{4}$, the $4\times 4$ identity matrix, and we have
\[
L=\mathscr{E}+LJ^{2},\quad \mathscr{E}(LJ^{2})=0,\quad \text{and}\quad I_{6}=I_{4}+(LJ^{2})^{2}.
\]
On the other hand, if $\mathcal{C}:=\operatorname{span}_{K} \{\mathscr{E},I_{4}\}$, then $\mathcal{C}$ is a subalgebra and 
\[
\mathcal{C}\cong K[x]/(x^{2}-1)\cong K\oplus K 
\]
where the last direct sum is as algebras. Therefore, it follows that 
\[
\mathcal{B}\cong K\oplus K\oplus M_{2}(K)\quad (\text{direct sum as algebras}).
\]
We conclude $J(\mathcal{A}/\mathcal{I})=0$. Therefore $\mathcal{I}=J(\mathcal{A})$. In this case, since $\langle J,LJ^{2}\rangle J(A)=0$ the algebra $\mathcal{A}$ is isomorphic to $\begin{pmatrix}
    B & 0\\
    0 & M_{2}(K)
\end{pmatrix}$, where $B=K\oplus K+\operatorname{span}_{K}\{M,ML\}$.  
\end{Example}

With the illustration provided by the Example \ref{Example.Jacobson}, and by the discussion preceding it, we may henceforth assume, without loss of generality, that $A=D_{1}\oplus \cdots \oplus D_{k}+J(A)$ is a regular $G$-graded algebra with bicharacter $\beta$ and that if $J(A)\neq 0$, then $D_{i}J(A)\neq 0$ for all $1\leq i\leq k$. In Section~ 6, it will become clear from the analysis how to deal with the case in which $A$ is only $\beta$-commutative.

\begin{Remark} 
\label{min.}
\begin{enumerate}
    \item  If the regular decomposition of $A$ is minimal, by Theorem \ref{charac.minimal}, it follows that $\det M^{A} \neq 0$ and $\exp(A)=|G|$. It is worth noting the cocycle $\alpha$ cannot be trivial, otherwise, $\beta\equiv 1$ and $\det M^{A}=0$. 

We know that for every $1 \leq j \leq k$, the cocycle $\alpha_{j}$ induces $\beta$, then $\det M^{K^{\alpha_{j}}G} \neq 0$. Therefore, the regular decomposition of $K^{\alpha_{j}}G$ is minimal and by Remark \ref{charc.minimal.twisted.1}, $G_{0}(\alpha_{j})=\{0\}$.

\item Since $A$ is in particular $\beta$-commutative, given $i\neq j$, because $D_{i}D_{j}=0$, it follows that $D_{i}J(A)D_{j}=0$.
\end{enumerate}
\end{Remark}

We recall that a commutative algebra $S$ is \textsl{local}, if $S/J(S)\cong K$, where $J(S)$ is the Jacobson radical of $S$. We have the following.

\begin{Corollary} 
\label{necess.condition}  $A\cong K^{\alpha}G+J(A)$, if and only if $A_{0}$ is a local algebra.
\end{Corollary}
\begin{proof}
  Suppose $A\cong K^{\alpha}G+J(A)$ and consider the Wedderburn-Malcev decomposition of $A_{0}=B+ J(A_{0})$, where $B\cong K^{\oplus m}$, with $m>1$. Let $D$ be a $G$-graded simple subalgebra of $A$ isomorphic to $K^{\alpha}G$. Suppose there exists a non-invertible homogeneous element $x_{0} \in B$, then $x_{0}=y+z$, where $y\in D$ and $z\in J(A)$. Since $x_{0}\in B\subseteq A_{0}$ we must have $y\in D_{0}$ and $z\in J(A)_{0}\underset{(\ref{eq.radical})}{=}J(A_{0})$. Since $D_{0}\cong K$, up to a constant, we can assume $x_{0}=1+z$, where $z\in J(A_{0})$. Thus, $x_0$ is an invertible element in $A_0$ that is a contradiction. It turns out $m=1$, i.e., $A_0$ is a local algebra.

  Conversely, suppose that $A_{0}=K+ J(A_{0})$. By Proposition \ref{sum.twisted}, there exist graded simple subalgebras $D_{1}$, \dots, $D_{k}$ of $A$ such that $D_{i} \cong K^{\alpha}H_{i}$, $\alpha_{i}\in H^{2}(G,K^{\ast})$, for each $1 \leq i \leq k$. Also, again by Proposition \ref{sum.twisted} there exists $D_{i}\cong K^{\alpha_{i}}G$. 
  Assume $k>1$ and let $\{d_{g}^{(1)} \mid g \in H_{1}\}$ and $\{d_{g}^{(2)} \mid g \in H_{2}\}$ be bases of $D_{1}$ and $D_{2}$, respectively, with $d_{g}^{(i)} d_{h}^{(i)}=\alpha(g, h)d_{g+h}^{(i)}$ for $i \in \{1, 2\}$ and $g$, $h\in H_{i}$. It is clear that $d_{0}^{(1)}$, $d_{0}^{(2)} \in K$ because $(D_{i})_{0}\cong K$, $i\in \{1,2\}$. It implies the existence of $\lambda \in K$ such that $d_{0}^{(1)} = \lambda d_{0}^{(2)}$, that is a contradiction. We conclude that $i=1$, thus $A \cong K^{\alpha}G +J(A)$. 
\end{proof}

Finite-dimensional regular graded algebras whose neutral component is local do not necessarily have to be isomorphic to a twisted group algebra, even when the regular decomposition is minimal. This is well illustrated by the following two examples, which rely on the fact that a twisted group algebra (over an algebraically closed field of characteristic $0$) is semisimple (see \cite[Theorem 3.3.6]{Karpilovsky.vol.2}).

\begin{Example}Let $A=K[z,t]$ be the polynomial algebra in two commutative variables. Denote by $I$ the ideal of $A$ generated by the set $\{z^{2}, t^{2}-1\}$. Writing $x$ instead of $\pi(x)$, where $\pi\colon A\rightarrow A/I$ is the canonical projection and $x\in A$, the quotient algebra $B:=A/I$ satisfies the relations $z^{2}=0$ and $t^{2}=1$. It follows that $\{1,z,t,zt\}$ forms a basis for $B$. Now, if we set $B_{0}=\text{\rm span}_{K}\{1,z\}$ and $B_{1}=\text{\rm span}_{K}\{t,zt\}$, then $B=B_{0}\oplus B_{1}$, and it defines a $\mathbb{Z}_{2}$-grading on $B$. We note that for any $m\in \mathbb{N}$
\[
t^{m}=\begin{cases}
    1,\quad \text{if}\quad m\quad \text{is even}\\
    t,\quad \text{if}\quad m\quad \text{is odd}
\end{cases}.
\]
Therefore, $B$ is a finite dimensional $\mathbb{Z}_{2}$-graded algebra with regular grading and bicharacter $\beta\equiv 1$. Moreover, $A_{0}=K+J(A_{0})$, where $J(A_{0})=\text{\rm span}_{K}\{z\}$ and $J(A)=\text{\rm span}_{K}\{z,zt\}$. The algebra $B$ cannot be isomorphic to a twisted group algebra because $J(A)\neq 0$. 
\end{Example}
\begin{Example} Let $\{0\}$ be the trivial abelian group, and let $e_{i,j}$ denote the elementary matrices of $M_{2}(K)$. Define the algebra
\[
U_{T}:=\operatorname{span}_{K}\{e_{1,1}+e_{2,2},e_{1,2}\}
=\Big\{\begin{pmatrix}
    a & b\\
    0 & a
\end{pmatrix}\mid a,b\in K\Big\}.
\]
Then, by Remark \ref{trivial.case}, the $\{0\}$-graded algebra $U_{T}=(U_{T})_{0}$ is a regular algebra with minimal regular decomposition, and it is a local algebra because $U_{T}=K(e_{1,1}+e_{1,2})+J(U_{T})$, where $J(U_{T})=Ke_{1,2}$. Because $J(U_{T})\neq 0$, $U_{T}$ cannot be isomorphic to a twisted group algebra.
\end{Example}

\begin{Lemma} 
\label{key.class.}  Assume $A_{0}$ is a local algebra. Then, if $J(A_{0})\neq 0$, we have
\[
J(A) \cong K^{\alpha}G \otimes J(A_{0}),
\]
where $J(A_{0})$ is considered with the trivial grading.

\end{Lemma}
\begin{proof} Since $A_{0}$ is local, by Corollary \ref{necess.condition}, $A\cong K^{\alpha}G+J(A)$. Let $D$ be a graded simple subalgebra of $A$ isomorphic to $K^{\alpha}G$ and denote by $\mathcal{D}=\{d_{g}\mid g\in G\}$ a basis of $D$ satisfying $d_{g}d_{h}=\alpha(g,h)d_{g+h}$, for any $g$, $h\in G$. We claim that $DJ(A_{0})=J(A)$. Indeed, since $J(A_{0})\subseteq J(A)$ and $J(A)$ is a graded two-sided ideal, we obtain that $DJ(A_{0})\subseteq J(A)$. Thus, it remains to prove  $J(A)\subseteq DJ(A_{0})$. Let $g\in G$ and $z_{g}\in J(A)_{g}$, then
\[
z_{g}=d_{g}d_{g}^{-1}z_{g}=d_{g}(d_{g}^{-1}z_{g});
\]
since $d_{g}^{-1}z_{g}\in J(A)_{0}\underset{(\ref{eq.radical})}{=}J(A_{0})$, we conclude $J(A)_{g}=D_{g}J(A_{0})$. Thus,
\[
J(A)=\oplus_{g\in G}J(A)_{g}=\oplus_{g\in G}D_{g}J(A_{0})\underbrace{=}_{J(A_{0})\subseteq A_{0}}DJ(A_{0}).
\]
Observe that $D\cap J(A_{0})=\{0\}$ because $D_{0}=\text{\rm span}_{K}\{d_{0}\}$ and $d_{0}$ is invertible. Consider the only linear function $\psi\colon D\otimes J(A_{0})\rightarrow J(A)$ satisfying $\psi (d\otimes a)=da.$ We show that $\psi$ is an isomorphism of graded rings, where $J(A_{0})$ is considered with the trivial $G$-grading: $J(A_{0})_{0}=J(A_{0})$ and $J(A_{0})_{g}=0$ if $g\neq 0$. If $d\otimes a$, $d'\otimes a'\in D\otimes J(A_{0})$ then, since $J(A_{0})\subseteq A_{0}$, it follows 
\[
\psi((d\otimes a)(d'\otimes a))=\psi(dd'\otimes aa')=dd'aa'\underbrace{=}_{J(A_{0})\subseteq A_{0}}(da)(d'a')=\psi(d\otimes a)\psi(d'\otimes a').
\]
On the other hand, given $g\in G$, if $d\otimes a\in (D\otimes J(A_{0}))_{g}=D_{g}\otimes J(A_{0})$ then 
\[
\psi(d\otimes a)=da\in D_{g}J(A_{0})=(J(A))_{g}.
\]
Thus, $\psi$ is a homomorphism of graded rings. Since $DJ(A_{0})=J(A)$ we conclude that $\psi$ is surjective. Now we show that $\psi$ is injective. Let $z\in D\otimes J(A_{0})$ satisfy $\psi(z)=0$. If $\{a_{1},\ldots, a_{k}\}$ is a basis of $J(A_{0})$, $\{d_{g}\otimes a_{1},\ldots, d_{g}\otimes a_{k}\}$ is a basis for $(D\otimes J(A_{0}))_{g}$, for every $g\in G$. Therefore, we can write $z$ as 
\[
z=\sum_{g,j}\gamma_{g,j}d_{g}\otimes a_{j},\quad \text{where}\quad \gamma_{g,j}\in K .
\]
Thus $\psi(z)=0$ if and only if $\sum_{g,j}\gamma_{g,j}d_{g} a_{j}=0$, that is, for every $g\in G$, we have
\[
\sum_{j}\gamma_{g,j}d_{g} a_{j}=0.
\]
The last expression vanishes if and only if
\[
 d_{g}(\sum_{j}\gamma_{g,j}a_{j})=0,
\]
but since $d_{g}$ is invertible we conclude $\sum_{j}\gamma_{g,j}a_{j}=0$, that is, $\gamma_{g,j}=0$ for every $1\leq j\leq k$. Hence, $\psi$ is injective and we are done. 
\end{proof}

The last result turns out to be crucial in the proof of the next theorem.

\begin{Theorem}
\label{class.1}
 Assume $A_{0}$ is a local algebra. Then, if $J(A_{0})\neq 0$, there exist $\mathscr{U}$ a commutative local subalgebra of $A$ and $\alpha\in H^{2}(G,K^{\ast})$ be a cocycle that induces $\beta$ such that $A\cong K^{\alpha}G\otimes \mathscr{U}$ as $G$-graded algebras (here we consider $\mathscr{U}$ with the trivial grading). Moreover, if $J(A_{0})=0$, then $A\cong K^{\alpha}G$. 
    
\end{Theorem}
\begin{proof} 
  It follows by Lemma \ref{necess.condition} that $A\cong K^{\alpha}G+J(A)$ as graded algebras. Thus 
 \[
 A\cong K^{\alpha}G+J(A)\cong K^{\alpha}G\otimes K+J(A)
 \]
 and by Lemma \ref{key.class.} we have $J(A)\cong K^{\alpha}G\otimes J(A_{0})$, hence if $\mathscr{U}:=K+J(A_{0})$ we get
 \[
 A\cong K^{\alpha}G\otimes K+ K^{\alpha}G\otimes J(A_{0})\cong  K^{\alpha}G\otimes (K+ J(A_{0}))=K^{\alpha}G\otimes \mathscr{U}.
 \]
 Finally, by the proof of Lemma \ref{key.class.}, given $g \in G$ and $z \in J(A)_{g}$, we can express $z$ as $z = d_{g}t$, where $t \in J(A_{0})$. If $J(A_{0})=0$, then $J(A)=0$, consequently, $A \cong K^{\alpha}G$.
\end{proof}

Now we consider the case where $A_{0}$ is not necessarily local. By the Wedderburn-Malcev theorem we write $A_{0}=E_{1}\oplus \cdots \oplus E_{l}+ J(A_{0})$, where $E_{i}\cong K$, for every $1\leq i\leq l$. Denote by $1_{E_{i}}$ the identity of $E_{i}$, $1\leq i\leq l$, then we have $1=1_{E_{1}}+\cdots +1_{E_{l}}$. 

\begin{Lemma} 
\label{cota.}
Assume that $A$ has complete support. Write $A=D_{1}\oplus \cdots \oplus D_{k}+J(A)$, where $D_{i}\cong K^{\alpha}G$, for every $1\leq i\leq k$, and $\alpha$ induces $\beta$. For each $1\leq i\leq k$, let $\mathcal{D}_{i}:=\{d_{g}^{(i)} \mid g \in G\}$ be a basis of $D_{i}$ satisfying 
\[
d_{g}^{(i)} d_{h}^{(i)}=\alpha(g, h) d_{g+h}^{(i)},\quad \text{for any}\quad g,h\in G.
\]
Then the following conditions hold:
\begin{enumerate}
    \item For every $1\leq i\leq k$, $d_{0}^{(i)}$ belongs to $E_{r_{i}}$, for some $1\leq r_{i}\leq l$.
    \item  $k=l$.
    \item If $J(A_{0})\neq 0$, then $E_{r_{i}}J(A_{0})\neq 0$ for every $1\leq i\leq k$.
\end{enumerate}
\end{Lemma}
\begin{proof} For each $1\leq i\leq k$, $D_{i}$ is a twisted group algebra isomorphic to $K^{\alpha}G$, in particular $D_{i}$ is a $G$-graded regular algebra with bicharacter $\beta$. We already know that $(D_{i})_{0}\cong K$. On the other hand, $(D_{i})_{0}\subseteq A_{0}$, then  $(D_{i})_{0}$ must be equal to $E_{r_{i}}$, for some $1\leq r_{i}\leq l$. In particular, $d_{0}^{(i)}\in E_{r_{i}}$.

Now, since $D_{1} \oplus \cdots \oplus D_{k}$ is a direct sum of algebras, it follows that for any $1 \leq i \neq j \leq k$, we have $r_{i} \neq r_{j}$. This implies that the map $i \mapsto r_{i}$ is injective, therefore $k \leq l$. Suppose $k < l$. The elements $d_{0}^{(1)}$, \dots, $d_{0}^{(k)}$ can be chosen to satisfy $1 = d_{0}^{(1)}+\cdots + d_{0}^{(k)}$, i.e., $1\in E_{r_{1}}\oplus \cdots \oplus E_{r_{k}}$. However, since $1 = 1_{E_{1}} + \cdots + 1_{E_{l}}$, if $\{r_{1}, \ldots, r_{k}, j_{1}, \ldots, j_{p}\} = \{1, \ldots, l\}$ with $k+ p=l$, then by uniqueness we would have $d_{0}^{(i)} = 1_{E_{r_{i}}}$ for any $1 \leq i \leq k$, and 
\[
1_{E_{j_{1}}} + \cdots + 1_{E_{j_{p}}} = 1 - (1_{E_{r_{1}}} + \cdots + 1_{E_{r_{k}}}) = 1 - 1 = 0,
\]
that is $1_{E_{j_{1}}} = \cdots = 1_{E_{j_{p}}} = 0$, which is a contradiction. Thus, $k = l$. In particular, $i\mapsto r_{i}$ is a bijection. 

Finally, we will show that $E_{r_{i}}J(A_{0})\neq 0$ for every $1\leq i\leq k$. Recall that, by hypothesis on $A$,  $D_{i}J(A)\neq 0$, for each $1\leq i\leq k$. Let $1\leq i\leq k$ and take $t\in D_{i}$ and $a\in J(A)_{g}$ (for some $g\in G$) such that $ta\neq 0$. If $1_{E_{r_{i}}}a= 0$, then $d_{h}^{(i)}a=0$, for any $h\in G$ and so, $ta=0$, that is a contradiction. Thus, $1_{E_{r_{i}}}a\neq 0$. In particular, $1_{E_{r_{i}}}(d_{g}^{(i)})^{-1}a=(d_{g}^{(i)})^{-1}a\neq 0$, i.e, $E_{r_{i}}J(A_{0})\neq 0$.
\end{proof}

\begin{Lemma}
\label{general.key.class}
Assume that $A$ has complete support. Suppose $J(A_{0})\neq 0$, then,  considering $J(A_{0})$ with the trivial grading we have
\[
K^{\alpha}G\otimes J(A_{0})\cong J(A) 
\]
as $G$-graded rings.
\end{Lemma}
\begin{proof} Here we will use the same notations as in the proof of Lemma \ref{cota.}. Also, by item (1) in the same lemma, we assume $d_{0}^{(i)}=1_{E_{r_{i}}}$, for every $1 \leq i \leq k$. By Remark \ref{min.}, for each $1 \leq i \neq j \leq k$, we have $D_{i}J(A)D_{j}=0$, and by item (3) in the Lemma \ref{cota.}, we have $E_{r_{i}}J(A_{0})\neq 0$ for any $1\leq i\leq k$. 

Given $g\in G$, take $z\in J(A)_{g}$, then
\[
z=1z=(1_{E_{r_{1}}}+\cdots +1_{E_{r_{k}}})z=1_{E_{r_{1}}}z+\cdots +1_{E_{r_{k}}}z,
\]
therefore
\[
z=d_{g}^{(1)}((d_{g}^{(1)})^{-1}z)+\cdots +d_{g}^{(k)}((d_{g}^{(k)})^{-1}z).
\]
Since $(d_{g}^{(j)})^{-1}z\in J(A)_{0}\underset{(\ref{eq.radical})}{=}J(A_{0})$, for any $1\leq j\leq k$, we get 
\[
J(A)_{g}\subseteq \oplus_{i=1}^{k}(D_{i})_{g}J(A_{0}).
\]
On the other hand, since $J(A_{0})\subseteq J(A)$ and $J(A)$ is a graded ideal, for any $1\leq i\leq k$, we have $(D_{i})_{g}J(A_{0})\subseteq J(A)_{g}$. Hence, we get
\[
J(A)_{g}=\oplus_{i=1}^{k}(D_{i})_{g}J(A_{0}).
\]
We obtain
\[
J(A) = \oplus_{g \in G} J(A)_{g} = \oplus_{g \in G} \bigg( \oplus_{i=1}^{k} (D_{i})_{g} J(A_{0}) \bigg) 
= \oplus_{i=1}^{k} \bigg( \oplus_{g \in G} (D_{i})_{g} \bigg) J(A_{0}) = \oplus_{i=1}^{k} D_{i} J(A_{0}).
\]
For every $1\leq i\leq k$, consider $J(A)^{(i)}:=D_{i}J(A_{0})$, then 
\[
J(A)=\bigoplus_{i=1}^{k}J(A)^{(i)}
\]
 and $J(A)^{(i)}$ is the Jacobson radical of the finite dimensional $G$-graded algebra with regular grading
\[
W_{i}:= D_{i}\oplus J(A)^{(i)}.
\]
However,
 \[
 (J(W_{i}))_{0}=J(A_{0})\cap J(A)^{(i)}=1_{E_{r_{i}}} J(A_{0}),
 \]
  and $W_{i}$ is a local algebra because $ (W_{i})_{0}\cong  E_{r_{i}}\oplus 1_{E_{r_{i}}} J(A_{0})$.
 Thus, by Lemma \ref{key.class.} we have
 \[
 J(A)^{(i)}\cong D_{i}\otimes 1_{E_{r_{i}}} J(A_{0}).
 \]
 Therefore
\[
J(A)=\oplus_{i=1}^{k}J(A)^{(i)}
     \cong \oplus_{i=1}^{k}D_{i}\otimes 1_{E_{r_{i}}} J(A_{0})\cong \oplus_{i=1}^{k}(K^{\alpha}G)\otimes 1_{E_{r_{i}}}J(A_{0})
\]
and so
\begin{equation}
\label{eq.aux.}
J(A)\cong K^{\alpha}G\otimes \bigg(\oplus_{i=1}^{k}1_{E_{r_{i}}}J(A_{0}) \bigg).    
\end{equation}

Now, if $V:=1_{E_{r_{1}}}J(A_{0})+\cdots+1_{E_{r_{k}}}J(A_{0})$, then since $A_{0}$ is commutative and $J(A_{0})$ is a nilpotent ideal of $A_{0}$, we conclude $V$ is a nilpotent ideal of $A_{0}$ contained in $J(A_{0})$. On the other hand, given $a\in J(A_{0})$, it can be written $a=1a=(1_{E_{r_{1}}}a)+\cdots+(1_{E_{r_{k}}}a)\in V$, i.e, $V=J(A_{0})$.

Then, by (\ref{eq.aux.}) we get
\[
J(A)\cong K^{\alpha}G\otimes J(A_{0}). \qedhere
\]

\end{proof}

\begin{Remark}
\label{iso.rapi.}
Let us notice a simple but important fact. Consider $A=D_{1}\oplus \cdots \oplus D_{k}+J(A)$ as in Lemma \ref{general.key.class}. Given $1\leq i\leq k$, any element of $E_{r_{i}}$ is of the form $1_{E_{r_{i}}}s$, $s\in K$, and so, we can define
\[
\phi'\colon D_{i}\times E_{r_{i}}\rightarrow D_{i},\quad (d,1_{E_{r_{i}}}s)\mapsto ds.
\]
It is easy to see that $\phi'$ is a bilinear map. Thus, by the universal property of the tensor product, there exists a linear map $\phi\colon D_{i}\otimes E_{r_{i}}\rightarrow D_{i}$ satisfying $\phi(d\otimes 1_{E_{r_{i}}}s)=ds$, for all $d\in D_{i}$ and $s\in K$. Since 
\[
\phi(d\otimes 1_{E_{r_{i}}})=d,\quad \text{for any}\quad d\in D_{i},
\]
we conclude $\phi$ is surjective. Finally, because of $\dim E_{r_{i}}=1$, we obtain $\dim  (D_{i}\otimes E_{r_{i}})=\dim D_{i}$. In particular, $\phi$ is a bijection, i.e, $ D_{i}\otimes E_{r_{i}}\cong D_{i}$. Moreover, if one considers $E_{r_{i}}$ with the trivial grading, $\phi$ turns out to be an isomorphism of graded algebras.

\end{Remark}

\begin{Theorem}
\label{main.theorem}
 Assume that $A$ has complete support.  Suppose that $J(A_{0})\neq 0$. Then, there exists $\mathscr{V}$ a commutative algebra such that the following conditions are satisfied
\begin{enumerate}
    \item Considering $\mathscr{V}$ with the trivial grading, we have that, as $G$-graded algebras
            \[
            A\cong K^{\alpha}G\otimes \mathscr{V} .
            \]
    \item  Let $P$ the maximal semisimple subalgebra of $A_{0}$ such that $A_{0}=P+J(A_{0})$. Then if $\dim P=1$ and $J(A_{0})=0$, it follows that $A\cong K^{\alpha}G$. The converse is also true.
    \end{enumerate}
Furthermore, if $J(A_{0})=0$, then $A\cong (K^{\alpha}G)^{\oplus k}.$
\end{Theorem}
\begin{proof} As we saw above, there exists $k\in \mathbb{N}$ such that $A=D_{1}\oplus \cdots \oplus D_{k}+J(A)$, where $D_{i}\cong K^{\alpha}G$ for each $1\leq i\leq k$. 

 By Lemma \ref{general.key.class} we have
\[
J(A)\cong K^{\alpha}G\otimes J(A_{0})\quad (\text{as $G$-graded rings}).
\]
 Define
\[
\mathscr{V}:=E_{r_{1}}\oplus \cdots \oplus E_{r_{k}}+ J(A_{0})\quad (\text{direct sum of vector spaces})
\]
and consider $\mathscr{V}$ with the trivial $G$-grading: $\mathscr{V}_{0}=\mathscr{V}$, and $\mathscr{V}_{g}=\{0\}$ if $g\neq 0$. We notice by item (2) in Lemma \ref{cota.} that $\mathscr{V}\cong A_{0}$.

Since $E_{r_{1}} \cong \cdots \cong E_{r_{k}} \cong K$, we get
\begin{align*} 
A &= D_{1}\oplus \cdots \oplus D_{k} +J(A)\underset{(\ref{iso.rapi.})}{\cong} (D_{1}\otimes E_{r_{1}})\oplus \cdots \oplus (D_{k}\otimes E_{r_{k}})+J(A)\\
&\cong (K^{\alpha}G\otimes E_{r_{1}})\oplus \cdots \oplus (K^{\alpha}G\otimes  E_{r_{k}})+J(A)\underset{(\ref{general.key.class})}{\cong} K^{\alpha}G\otimes (E_{r_{1}}\oplus \cdots \oplus E_{r_{k}})+K^{\alpha}G\otimes J(A_{0}) \\
& \cong K^{\alpha}G\otimes (E_{r_{1}}\oplus \cdots \oplus E_{r_{k}}+J(A_{0}))\cong K^{\alpha}G\otimes \mathscr{V}
\end{align*}
and item (1) follows. For  item (2), notice that if $P$ is a maximal semisimple subalgebra of $A_{0}$ such that $A_{0}=P+J(A_{0})$, then by uniqueness, up to  conjugation, of the Wedderburn-Malcev theorem \cite[Theorem 3.4.3]{GZbook}, it follows that 
\[
P\cong E_{1}\oplus \cdots \oplus E_{k}.
\]
 Thus, we have $\dim P= k$.  Now, by Corollary \ref{necess.condition} and Theorem \ref{class.1}, we conclude if $\dim P=1$ and $J(A_{0})=0$, then $A\cong K^{\alpha}G$. The converse is clear, because $K^{\alpha}G$ is semisimple, and $(K^{\alpha}G)_{0}\cong K$. Finally, recall that by the proof of Lemma \ref{general.key.class}, for any $g\in G$, and $z \in J(A)_{g}$, we can express it as  
\[
z=d_{g}^{(1)}t_{0,1}+\cdots+d_{g}^{(k)}t_{0,k},
\]  
where $t_{0,j} \in J(A_{0})$, $1\leq j\leq k$. Consequently, if $J(A_{0}) = 0$, then $J(A) = 0$, implying that $A \cong (K^{\alpha}G)^{\oplus k}$.
\end{proof}

\section{An explicit computation of the graded codimensions}

In this section, we will use Theorem \ref{main.theorem} in order to explicitly calculate the codimension sequence of a finite dimensional $G$-graded algebra $A$ with bicharacter $\beta$, complete support, and minimal regular decomposition. Moreover, if we write $A=D_{1}\oplus \cdots \oplus D_{k}+J(A)$ , where $D_{i}$ is a $G$-graded simple subalgebra of $A$, then $D_{i}J(A)\neq 0$, for all $1\leq i\leq k$.  In particular, it will be shown that the graded PI-exponent of such algebras coincides with the ordinary PI-exponent. 

We shall recall some basic notions. Let $Q$ be any finite group and consider the free $Q$-graded algebra 
\[
K\langle X_{Q}\rangle=K\langle x_{i}^{(q)}\mid q\in Q,\quad i\in \mathbb{N}\rangle.
\] 
Given $n\in \mathbb{N}$, let $\mathbf{q}=(q_1,\ldots,q_n)\in Q^{n}$. Consider the following subspace of $K\langle X_{Q}\rangle$
\[
P_{\mathbf{q}}=\text{\rm span}_{K}\{x_{\sigma(1)}^{(q_{\sigma(1)})}\cdots x_{\sigma(n)}^{(q_{\sigma(n)})}\mid \sigma\in S_{n}\}
\]
and define
\[
P_{n}^{Q}:=\oplus_{\mathbf{g}\in Q^{n}} P_{\mathbf{q}}.
\]
In this way, if $R=\oplus_{q\in Q} R_{q}$ is a $Q$-graded algebra, we define the $n$-th $Q$-{\it graded codimension} of $R$ by
\[
c_{n}^{Q}(R)=\dim_{K}\dfrac{P_{n}^{Q}}{P_{n}^{Q}\cap T_{Q}(R)}.
\]


Let $Q$ be a finite group and $R$ be a $Q$-graded algebra. 

\begin{enumerate}
    \item[(a)] Let $c_{n}(R)$ be the ordinary $n$-th codimension of $R$, then:
    \begin{enumerate}
        \item[(a.1)] $c_{n}(R)\leq c_{n}^{Q}(R)$ \cite[Lemma 3.1]{BGH},
        \item [(a.2)]  If $R$ is a PI-algebra, then $c_{n}^{Q}(R)\leq |Q|^{n}c_{n}(R)$ \cite[p. 268]{GZbook};
    \end{enumerate}
    
    \item[(b)] the limit
    \begin{equation}
    \label{basic.defi.}
        \exp^{Q}(R)=\lim_{n\rightarrow \infty} \sqrt[n]{c_{n}^{Q}(R)}
    \end{equation}
exists and is a non-negative integer \cite[Theorem 2.3]{AG}. 
\end{enumerate}

\begin{Definition}
    The integer in item (b) above is called the \textsl{graded PI-exponent of $R$}. 
\end{Definition}
    
\begin{Remark} 
It is worth briefly describing the main works related to the graded PI-exponent in the context of associative algebras over fields of characteristic zero. In 1999, Giambruno and Zaicev showed in \cite{GZ} the existence and integrality of the ordinary PI exponent (or graded by a trivial group), thus providing a positive answer to Amitsur’s conjecture. Later on, in 2003, Benanti, Giambruno, and Pipitone extended this result in \cite{Pipitone}, considering algebras graded by $\mathbb{Z}_{2}$. In 2009, Giambruno and La Mattina further established in \cite{LaMattina} the existence of the graded PI exponent for algebras graded by finite abelian groups. Finally, the case of algebras graded by arbitrary finite groups was solved in 2013 by Aljadeff and Giambruno in \cite{AG}.
\end{Remark}

Denote by $x_{i}^{(g)}$ and $v_{i}^{(g,h)}$ the free generators of $K\langle X_{G} \rangle$ and $K\langle X_{G \times G} \rangle$, respectively. Given $\beta \colon G \times G \rightarrow K^{\ast}$ a bicharacter  of $G$, let $K_{\beta}\langle X_{G} \rangle$ denote the relatively free algebra with bicharacter $\beta$ (Theorem \ref{free.reg.}), where the free generators of $K_{\beta}\langle X_{G} \rangle$ are denoted by $y_{i}^{(g)}$. 

Given $\textbf{h}=(h_1, \ldots, h_n) \in G^{n}$, we define a linear map $\phi_{\textbf{h}} \colon P_{n}^{G} \rightarrow P_{n}^{G \times G}$ as follows: Given $\tau \in S_{n}$, if in $K_{\beta}\langle X_{G}\rangle$ we have
\[
y_{\tau(1)}^{h_{\tau(1)}} y_{\tau(2)}^{h_{\tau(2)}} \cdots y_{\tau(n)}^{h_{\tau(n)}} = \mu(\textbf{h}, \tau) y_{1}^{h_{1}} y_{2}^{h_{2}} \cdots y_{n}^{h_{n}}
\]
with $\mu(\textbf{h},\tau) \in K^{\ast}$ denoting a product of terms $\beta(h_{i},h_{j})$, then we define $\phi_{\textbf{h}}$ on the monomials of $P_{n}^{G}$ as follows:
\[
\phi_{\textbf{h}} \big( x_{\tau(1)}^{(g_{\tau(1)})} x_{\tau(2)}^{(g_{\tau(2)})} \cdots x_{\tau(n)}^{(g_{\tau(n)})} \big) = \mu(\textbf{h}, \tau) v_{\tau(1)}^{(g_{\tau(1)}, h_{\tau(1)})} v_{\tau(2)}^{(g_{\tau(2)}, h_{\tau(2)})} \cdots v_{\tau(n)}^{(g_{\tau(n)}, h_{\tau(n)})}
\]
and extend it by linearity to all of $P_{n}^{G}$.   

\begin{Theorem}
\label{teo.app.}
Let $B$ be a $G$-graded algebra with regular grading, $S$ any $G$-graded algebra, and consider the $G$-graded algebra $L:=B\otimes S$.

\begin{enumerate}
    \item[(i)] Given $\textbf{h}=(h_1, \ldots, h_n) \in G^{n}$ and $f(x_{1}^{(g_{1})}, \ldots, x_{n}^{(g_{n})}) \in P_{n}^{G}$, it follows that $f \in T_{G}(S)$ if and only if $\phi_{\textbf{h}}(f) \in T_{G \times G}(L)$.  
    \item[(ii)] 
\[
P_{n}^{G \times G} \cap T_{G \times G}(L) = \sum_{\textbf{h} \in G^{n}} \phi_{\textbf{h}}(P_{n}^{G} \cap T_{G}(S)), \quad \textbf{h} = (h_{1}, \ldots, h_{n}).
\]
In particular,
\[
c_{n}^{G}(L)=|G|^{n}c_{n}^{G}(S).
\]
\end{enumerate}

\end{Theorem}
\begin{proof}
The proof of (i) is exactly the same as \cite[Theorem 3.1]{BD}, since the regularity of $B$ implies the existence of $a_{h_{1}} \in B_{h_{1}}$, \dots, $a_{h_{n}} \in B_{h_{n}}$ such that $a_{h_{1}} \cdots a_{h_{n}} \neq 0$. Similarly, the proof of (ii) is the same as the one given in \cite[Corollary 3.2]{BD}.
\end{proof}

\begin{Corollary}  For any $n\in \mathbb{N}$,  $c_{n}^{G}(A)=|G|^{n}$. In particular, $\exp^{G}(A)=\exp(A)=|G|$. 
\end{Corollary}
\begin{proof} By Theorem \ref{main.theorem} the algebra $A$ is $G$-graded isomorphic to $(K^{\alpha}G)\otimes \mathscr{V}$, where $\mathscr{V}$ is a commutative algebra and $\alpha$ is a cocycle inducing $\beta$. Then, Theorem \ref{teo.app.}(ii) implies that

\[
c_{n}^{G}(A)=c_{n}^{G}((K^{\alpha}G)\otimes \mathscr{V})=|G|^{n}c_{n}^{G}(\mathscr{V}).
\]
Since $\mathscr{V}$ is commutative and non-nilpotent, we get $c_{n}^{G}(\mathscr{V})=1$, therefore $c_{n}^{G}(A)=|G|^{n}$. On the other hand, by (\ref{basic.defi.}), we obtain 
\[
\exp^{G}(A)=\lim_{n\rightarrow \infty} \sqrt[n]{c_{n}^{G}(A)}=\sqrt[n]{|G|^{n}}=|G|\underset{(\ref{charac.minimal})}{=}\exp(A). \qedhere
\]
\end{proof}

\section{General case}

In this section, we consider the general case, namely, when the algebra $A$ does not necessarily have complete support.  As in the proof of Proposition \ref{sum.twisted} we write $A=D_{1}\oplus\cdots \oplus D_{k}+J(A)$, where $D_{1}\cong\cdots  \cong D_{p}\cong K^{\alpha}G$, $\alpha\in H^{2}(G,K^{\ast})$, and $D_{i}\cong K^{\alpha_{i}}H_{i}$, where $H_{i}$ is a proper subgroup of $G$ and $\alpha_{i}\in H^{2}(H_{i},K^{\ast})$ for all $p<i\leq k$. Also, we will assume $D_{i}J(A)\neq 0$ for all $1\leq i\leq k$. 

Notice if $p=0$ then $A$ is just a $\beta$-commutative algebra which is not regular and if $p=k$ then $A$ has complete support. Given that our analysis will implicitly cover the first case, we can assume without loss of generality $1\leq p<k$. 

From now on, we will use the same notation as in the proof of Lemma \ref{cota.}.

Notice the proof of item (2) of Lemma \ref{cota.} is entirely independent of the assumptions of complete support. Therefore, the same arguments yield the following.

\begin{Lemma} If $A_{0}=E_{r_{1}}\oplus \cdots \oplus E_{r_{l}}+J(A_{0})$ and $A=D_{1}\oplus\cdots \oplus D_{k}+J(A)$, then the map $\{1,\ldots, k\}\rightarrow \{r_{1},\ldots, r_{l}\}$ given by $i\mapsto r_{i}$, is a bijection. In particular, $k=l$.
\end{Lemma}

\begin{Lemma}
\label{Lemma.A}
If $A_{0}=E_{r_{1}}\oplus \cdots \oplus E_{r_{k}}+J(A_{0})$, then for each $1\leq i\leq k$, $(D_{i})_{0}=E_{r_{i}}$, for some $1\leq r_{i}\leq l$. In particular $d_{0}^{(i)}$ belongs to $E_{r_{i}}$, for all $1\leq i\leq k$.
\end{Lemma}
\begin{proof}  Since $(H_{i})_{0}\cong K$, the result follows as in item (1) of Lemma \ref{cota.}.
\end{proof}
\begin{Lemma}
\label{general.lemma}
 If $J(A_{0})\neq 0$, then  $E_{r_{i}}J(A_{0})\neq 0$, for all $1\leq i\leq p$.
\end{Lemma}
\begin{proof} The proof follows from the same argument used in the proof of item (3) in Lemma \ref{cota.}.
\end{proof}

\begin{Lemma} If $p>1$, then the cocycles $\alpha_{1}$,\ldots, $\alpha_{p}$ corresponding respectively to $D_{1}$, \dots, $D_{p}$, are equal. 
\end{Lemma}
\begin{proof}
    Because of \cite[Lemma 31]{Eli1}, we have $T_{G}(D_{i})=T_{G}(D_{j})$, for all $1\leq i,j\leq p$. Now, the proof follows exactly as in item (3) of Proposition \ref{sum.twisted}.
\end{proof}

Based on the results above, the decomposition $A=D_{1}\oplus \cdots \oplus D_{k}+J(A)$ satisfies the following: 
\begin{enumerate}
        \item $D_{i}J(A)\neq 0$, for all $1\leq i\leq k$.
        \item $(D_{i})_{0}=E_{r_{i}}$, for all $1\leq i\leq k$.
         \item  If $J(A_{0})\neq 0$, then $1_{E_{r_{j}}}J(A_{0})\neq 0$, for all $1\leq j\leq p$.
         
        \item $\alpha_{1}=\cdots=\alpha_{p}=\alpha$ and $D_{i}\cong K^{\alpha}G$, for all $1\leq i\leq p$.
        
\end{enumerate}

\begin{Remark} Given $p+1\leq i\leq k$, we  have $1_{E_{r_{i}}}J(A)\neq 0$. Indeed, since $D_{i}J(A)\neq 0$, there exists $d_{h}^{(i)}\in \mathcal{D}_{i}$ such that $d_{h}^{(i)}z\neq 0$, for some $z\in J(A)$. If $1_{E_{r_{i}}}z=0$, we would have
\[
d_{h}^{(i)}z=\alpha(h,0)^{-1}d_{h}^{(i)}1_{E_{r_{i}}}z=0
\]
that is a contradiction. 
\end{Remark}

\begin{Theorem} 
\label{main.general}

. 
\begin{enumerate}
    \item[(i)] Suppose $J(A_{0})\neq 0$. Then, there exist $\mathscr{V}$ a commutative algebra, and $\mathscr{W}$  a $G$-graded algebra which is not regular such that
\[
A\cong (K^{\alpha}G\otimes \mathscr{V})+\mathscr{W}\quad (\text{as $G$-graded algebras})
\]
and $J(\mathscr{V})+J(\mathscr{W}_{0})=J(A_{0})$.
\item[(ii)] If $J(A_{0})=0$, then $A$ is semisimple and it is a sum of twisted group algebra with at least one of its $G$-graded simple components isomorphic to $K^{\alpha}G$, for some $\alpha\in H^{2}(G,K^{\ast})$. 
\end{enumerate}

\end{Theorem}
\begin{proof}
We shall use arguments analogous to those of Lemma \ref{general.key.class}. Denote by $\mathcal{D}_{i}$ the canonical basis of $D_{i}$, $1\leq i\leq k$. Given $g\in \operatorname{Supp}(J(A))$ and $z\in J(A)_{g}$, we know that $g\in \operatorname{Supp}(D_{i})$, for all $1\leq i\leq p$. Now, we have
\[
z=1z=(1_{E_{r_{1}}}z+\cdots +1_{E_{r_p}}z)+(1_{E_{r_{p+1}}}z+\cdots+1_{E_{r_{k}}}z)
\]
thus, as in \ref{general.key.class} we can write
\[
z=(d_{g}^{(1)}((d_{g}^{(1)})^{-1})z)+\cdots+d_{g}^{(p)}((d_{g}^{(p)})^{-1})z))+(1_{E_{r_{p+1}}}z+\cdots+1_{E_{r_{k}}}z),
\]
hence
\[
z\in (D_{1})_{g}J(A_{0})\oplus \cdots \oplus (D_{p})_{g}J(A_{0})\oplus (D_{p+1})_{0}J(A)_{g}\oplus \cdots \oplus (D_{k})_{0}J(A)_{g}.
\]
In that way, we conclude
\[
J(A)=(D_{1})J(A_{0})\oplus \cdots \oplus (D_{p})J(A_0)\oplus (D_{p+1})_{0}J(A)\oplus \cdots \oplus (D_{k})_{0}J(A).
\]
Now, if $W_{i}:=D_{i}\oplus (D_{i})_{0}J(A)$, where $p+1\leq i\leq k$, then clearly $W_{i}$ is a $G$-graded algebra which is not regular and $J(W_{i})=(D_{i})_{0}J(A)$. Therefore, proceeding as in the proof of Lemma \ref{general.key.class}, we get
\[
J(A)\cong (K^{\alpha}G\otimes \widehat{J}_{0})+\big(\oplus_{i=p+1}^{k} J(W_{i})\big)
\]
where $\widehat{J}_{0}:=1_{E_{r_{1}}}J(A_{0})+\cdots+1_{E_{r_{p}}}J(A_{0})$. Consequently, in the same way as Theorem \ref{main.theorem}, if $\mathscr{V}:=E_{r_{1}}\oplus\cdots \oplus  E_{r_{p}}+\widehat{J}_{0}$ we have
\[
A\cong K^{\alpha}G\otimes \mathscr{V}+(W_{p+1}\oplus \cdots\oplus  W_{k})\quad (\text{as $G$-graded algebras}).
\]
 Thus, if we take $\mathscr{W}:=W_{p+1}\oplus \cdots \oplus W_{k}$, then the result follows. Now, by the definition of $\mathscr{W}$ and by the same argument used at the end of the proof of Lemma \ref{general.key.class}, we obtain
 \[
 J(\mathscr{W}_{0})=1_{E_{r_{p+1}}}J(A_{0})+\cdots +1_{E_{r_{k}}}J(A_{0})
 \]
 therefore, again by the end of proof of Lemma \ref{general.key.class} we get
 \[
J(\mathscr{V})+J(\mathscr{W}_{0})=1_{E_{r_{1}}}J(A_{0})+\cdots+1_{E_{r_{k}}}J(A_{0})=J(A_{0})
 \]
 and this proves (i). We observe at this point that, for some $p+1\leq i\leq k$, it may occur that $1_{E_{r_{i}}}J(A_{0})=0$. 

 Let us prove (ii). Suppose $J(A_{0})=0$. We know that there exists $d_{g}^{(1)}\in D_{1}$, with some $g\in G$, such that $d_{g}^{(1)}z\neq 0$, for some homogeneous element $z\in (J(A))_{\kappa}$, $\kappa\in G$. In particular, by multiplication by $(d_{g}^{(1)})^{-1}$, we conclude $1_{E_{r_{1}}}z\neq 0$. If for some $h\in G$ we have  $d_{h}^{(1)}z=0$, then $1_{E_{1}}z=0$, which is not the case. Therefore, $d_{h}^{(1)}z\neq 0$, for all $h\in G$, that implies $0\neq d_{-\kappa}^{(1)}z\in A_{0}\cap J(A_{0})=J(A_{0})$, which is a contradiction. We conclude $J(A)=0$, so $A$ is a direct sum of twisted group algebras with $D_{1}\cong K^{\alpha}G$, where $\alpha\in H^{2}(G,K^{\ast})$.
\end{proof}

\begin{Remark} The hypothesis that $D_{i}J(A)\neq 0$ cannot be removed in the Theorem \ref{main.general}. Indeed, considering the Example \ref{Example.Jacobson}, in that regular grading we have $J(\mathcal{A}_{0})=0$, but $J(\mathcal{A})\neq 0$. 
    
\end{Remark}

With the results of this section, together with the discussion preceding Example \ref{Example.Jacobson}, we obtain the main result of this work:

\begin{Theorem} 
\label{main.general.1}
Let $\mathcal{B}$ be a finite-dimensional $G$-graded regular algebra with bicharacter $\beta$. Then there exist $\mathcal{A}$ and $\mathcal{S}$, both $G$-graded subalgebras of $\mathfrak{B}$, with 
$\mathcal{S}=K^{\alpha_{1}}Q_{1}\oplus \cdots \oplus K^{\alpha_{t}}Q_{t}$, $\mathcal{B}=\mathcal{A}\oplus \mathcal{S}$ (direct sum as algebras) and $J(\mathcal{B})=J(\mathcal{A})$, such that the following hold:

\begin{enumerate}
    \item If $J(\mathcal{B})\neq 0$, then $\mathcal{A}\neq 0$ and the product of any graded simple component of $\mathcal{A}$ with the Jacobson radical of $\mathcal{B}$ is nonzero. 
    \item If $\mathcal{A}$ has non-graded simple components isomorphic to some $K^{\alpha}G$, then $\mathcal{S}$ must contain at least one isomorphic to $K^{\alpha}G$, and $\mathcal{A}$ is a $\beta$-commutative algebra that is not  a regular $G$-grading.

    \item If $\mathcal{A}$ has a graded simple component isomorphic to $K^{\alpha}G$, then we have the options: 
    \begin{enumerate}
        \item[(i)] If $J(\mathcal{B}_{0})\neq 0$, there exist $\mathscr{V}$, a commutative algebra, and $\mathscr{W}$, a $G$-graded algebra which is not regular, such that
\[
\mathcal{A}\cong (K^{\alpha}G\otimes \mathscr{V})+\mathscr{W} \quad (\text{as $G$-graded algebras})
\]
and $J(\mathscr{V})+J(\mathscr{W}_{0})=J(\mathcal{B}_{0})$.
        \item[(ii)] If $J(\mathcal{B}_{0})=0$, then $J(\mathcal{B})=0$. In particular, $\mathcal{A}$ is semisimple.
    \end{enumerate}
\end{enumerate}
    
\end{Theorem}

We finish this paper by giving another characterization of finite dimensional regular gradings with minimal regular decomposition.
\begin{Theorem}  Let $\mathcal{B}$ be a finite dimensional $G$-graded regular algebra with bicharacter $\beta$. Then, the regular decomposition of $\mathcal{B}$ is minimal if and only if $\exp(\mathcal{B})=|G|$.  
\end{Theorem}
\begin{proof} 

By Theorem \ref{charac.minimal} we just need to show that $\exp(\mathcal{B})=|G|$ implies the regular decomposition of $\mathcal{B}$ is minimal. Since $T(\mathcal{B})=T(K^{\alpha}G)$, for some $\alpha\in H^{2}(G,K^{\ast})$, it is enough to show that $\exp(K^{\alpha}G)=|G|$ implies that $K^{\alpha}G$ is a central simple algebra. By Yamazaki's Theorem \cite[Corollary 8.2.10]{Karpilovsky},  there exist $n$, $m\in \mathbb{N}$ such that 
\[
K^{\alpha}G\cong\underbrace{ M_{n}(K)\oplus\cdots\oplus M_{n}(K)}_{m}.
\]
and by the construction of the PI-exponent we have  $|G|=\exp(A)=\dim M_{n}(K)$. Since $\dim K^{\alpha}G=|G|$, we conclude $K^{\alpha}G\cong M_{n}(K)$ are we are done. 
\end{proof}


\bibliographystyle{abbrv}
\bibliography{Ob.reg.bib}

\end{document}